\documentclass[11pt]{amsart}

\usepackage{amsmath,amssymb,latexsym}
\usepackage{dsfont}
\usepackage[T1]{fontenc}
\usepackage{enumerate}
\usepackage{eqnarray}
\usepackage{mathtools}
\usepackage{url}
\usepackage{xcolor}
\usepackage{enumitem}

\usepackage[a4paper,top=3cm, bottom=3cm, inner=3.5cm,outer=2.7cm,nofoot,headsep=1.5cm]{geometry}
\def\Ind#1#2{#1\setbox0=\hbox{$#1x$}\kern\wd0\hbox to 0pt{\hss$#1\mid$\hss}
\lower.9\ht0\hbox to 0pt{\hss$#1\smile$\hss}\kern\wd0}
\def\Notind#1#2{#1\setbox0=\hbox{$#1x$}\kern\wd0\hbox to 0pt{\mathchardef
\nn="3236\hss$#1\nn$\kern1.4\wd0\hss}\hbox to 0pt{\hss$#1\mid$\hss}\lower.9\ht0
\hbox to 0pt{\hss$#1\smile$\hss}\kern\wd0}
\def\indi{\mathop{\mathpalette\Ind{}}}
\def\nindi{\mathop{\mathpalette\Notind{}}}

\theoremstyle{plain}
\newtheorem{theorem}{Theorem}[section]
\newtheorem{prop}[theorem]{Proposition}
\newtheorem{proposition}[theorem]{Proposition}

\newtheorem{fact}[theorem]{Fact}
\newtheorem{lemma}[theorem]{Lemma}
\newtheorem{cor}[theorem]{Corollary}

\theoremstyle{definition}
\newtheorem{defn}[theorem]{Definition}
\newtheorem{definition}[theorem]{Definition}
\newtheorem{remark}[theorem]{Remark}

\newtheorem{example}[theorem]{Example}

\newtheorem{notation-num}{Notation}

\newcommand{\tp}{\operatorname{tp}}

\newcommand{\cU}{\mathbb{M}}

\newcommand{\qftp}{\operatorname{qftp}}
\newcommand{\acl}{\operatorname{acl}}
\newcommand{\dcl}{\operatorname{dcl}}
\newcommand{\eq}{\operatorname{eq}}

\newcommand{\bdd}{\operatorname{bdd}}
\newcommand{\heq}{\operatorname{heq}}

\newcommand{\NTP}{\operatorname{NTP}}
\newcommand{\Autf}{\operatorname{Autf}}

\newcommand{\EM}{\operatorname{EM}}

\newcommand{\Aut}{\operatorname{Aut}}
\newcommand{\DLO}{\operatorname{DLO}}

\newcommand{\Th}{\operatorname{Th}}

\newcommand{\st}{\operatorname{st}}
\newcommand{\lex}{\operatorname{lex}}

\newcommand{\ind}{\operatorname{\indi}}
\newcommand{\nind}{\operatorname{\nindi}}

\newcommand{\dprk}{\operatorname{dprk}}

\newcommand{\im}{\operatorname{Im}}

\newcommand{\Q}{\mathbb{Q}}

\newcommand{\Lcal}{\mathcal{L}}

\newcommand{\TTr}{\operatorname{Tr}}
\newcommand{\TCirc}{\operatorname{Circ}}
\newcommand{\DT}{\operatorname{DTr}}
\newcommand{\DCirc}{\operatorname{DCirc}}
\newcommand{\eqdef}{\mathrel{:=}}
\newcommand{\circleT}{T_{\rm circ}}
\newcommand{\wideT}{\widehat T}

 \email{artem@umd.edu} 
 
\title{Forking in monadically NIP theories}
\author{Artem Chernikov}

\begin{document}
\begin{abstract}
We demonstrate that forking in monadically NIP theories satisfies a very strong form of triviality over arbitrary sets, generalizing Baldwin-Shelah in monadically stable theories and analogous to Shelah's result  for finite satisfiability over models in monadically NIP theories. This implies in particular that strict non-forking independence is particularly well behaved, and that the forking relation induces a canonical semilinear partial order on singletons. We also give an example of an $\omega$-categorical monadically NIP theory so that no finite set is an extension base.
 \end{abstract}
\maketitle

\section{Introduction}
A theory is \emph{monadically NIP} if every expansion of every model by
unary predicates is NIP. 			 The study of these theories goes back to
Baldwin--Shelah and Shelah
\cite{baldwin1985second,shelah2006monadic}; more recent characterizations
using finite satisfiability and indiscernible sequences include 
\cite{blumensath2011simplea, blumensath2011simple, braunfeld2021characterizations,braunfeld2024corrigenda,braunfeld2025indiscernibles}.
For deep and rapidly developing connections with finite model theory and algorithmic graph theory,
see e.g.~the survey \cite{pilipczuk2025lens}.

\emph{Forking} in general NIP theories has been studied extensively \cite{shelah2009dependent, adler2008introduction, hrushovski2011nip,chernikov2012forking, NTP2, kaplan2014strict, yaacov2014independence,simon2020amalgamation}. Unlike in stable theories, nonforking need not be
symmetric, and has other interesting issues and features. In this paper we show that forking is particularly well behaved restricting to monadic NIP theories.  

Our first result is that forking in a monadically NIP theory satisfies
the \emph{Baldwin--Shelah triviality} over every small set $A$
(Theorem~\ref{arb:forking-bs}): for any tuples $\bar{a},\bar{b}$ and singleton $c$,
\[
 \bar b\ind^f_A\bar a
 \quad\Longrightarrow\quad
 \bar bc\ind^f_A\bar a\ \text{ or }\ \bar b\ind^f_A\bar ac.
\]
We also prove an abstract version for bounded standard preindependence
relations over extension bases with left extension 
(Theorem~\ref{thm: dich for ind rel mon NIP}), which simultaneously generalizes  the results  for  forking in monadically stable theories in \cite{baldwin1985second} and for finite satisfiability over models in monadically NIP theories in \cite{shelah2006monadic}. In the same way as the analogous property for finite satisfiability $\ind^u$ is known to characterize 
monadic NIP \cite{braunfeld2021characterizations,braunfeld2024corrigenda}, this property of forking characterizes monadic NIP among all NIP theories (Corollary~\ref{conv:forking-characterization}). 

The Baldwin--Shelah triviality implies left total triviality
over arbitrary bases (Corollary~\ref{cor: arbitrary-left-total}). Over extension bases, it also gives pairwise criteria for
\emph{strict independence}, answering some open 
questions of Kaplan and Usvyatsov \cite{kaplan2014strict} in the monadically NIP case (Section \ref{sec: strict nf}).

In Section \ref{sec: fork forest} we consider a common generalization of the forking equivalence relation in monadically stable theories studied by Baldwin and Shelah \cite{baldwin1985second} and Shelah's partial order defined via $\ind^u$ over models in monadically NIP theories \cite{shelah2006monadic} (studied further in \cite{blumensath2011simple}). Namely, for arbitrary fixed small sets $C,D$, the relation
$a\trianglerighteq^fb$ defined by $a\nind^f_C Db$ is a preorder
on  singletons outside $\acl(C)$. Our main result is that its quotient partial order on
$Y :=\{a\notin\acl(C):a\ind^f_C D\}$ is ``tree like'', or more precisely \emph{semilinear}, i.e.~the set of predecessors of any element is linearly ordered;  and the
excluded points, when present, form one greatest class
(Proposition~\ref{strict:preorder}). This semilinear partial order contains
an infinite strict chain for some $C,D$ if and only if $T$ is unstable (Theorem~\ref{strict:instability-characterization}).

Finally, using the construction in \cite{estevan2021non}, we obtain an
$\omega$-categorical monadically NIP distal theory with no finite
forking-extension base. We also prove dp-minimality and distality for
arbitrary colored meet trees with linearly ordered open cones on the way 
(Theorem~\ref{thm: dp-min of conv ord trees}).

%
%
%
%
%
%
%
%
%
%
%
%
%
%
%
%
%
%
%
%
%
%
%
%

\subsection{Acknowledgements}
The results were mostly obtained in the Spring of 2026 and presented on several occasions during the Summer of 2026, including seminars at the Hebrew and Dresden universities, I thank the participants for their comments.  I also thank Eran Alouf, Manuel Bodirsky, Ioannis Eleftheriadis, Itay Kaplan, Chris Laskowski, Aris Papadopoulos and Garrett Drake Peters for some conversations that motivated us to write this note. I thank Pierre Simon and Szymon Toru\'nczyk for discussions related to Section \ref{sec: convexly ordered trees} many years ago. Chernikov was partially supported by the NSF
Research Grants DMS-2246598, DMS-2554164 and by the Alexander von Humboldt Foundation.

\section{Preliminaries}

Throughout the paper, $a,b,c, \ldots$ will denote small tuples and $A,B,C, \ldots$ small subsets of a sufficiently saturated and homogeneous monster model $\cU \models T$.

\subsection{Preindependence relations and forking in NIP theories}

In this section we review some pre-independence relations and their basic properties, in general and in NIP theories.

\begin{definition}
	We will consider the following properties of a relation $\ind$ on triples of small subsets of $\cU$.
	\begin{enumerate}
	\item (Automorphism) Invariance: if $a \ind_{A} b$ and $\sigma \in \Aut(\cU)$, then $\sigma(a) \ind_{\sigma(A)} \sigma(b)$.
	\item Monotonicity: if $a a' \ind_{A} b b'$ then $a \ind_{A} b$.
	\item (Strong) finite character: if $a \nind_{A} b$, then there is a formula $\varphi(x) \in \tp(a/Ab)$ so that for any $a' \models \varphi(x)$ we have $a' \nind_{A} b$.
	\item Base monotonicity: if $a \ind_{A} bc$ then $a \ind_{A b} c$.
	\item Left transitivity (over $A$): $a \ind_{A b} c$ and $b \ind_{A} c$ implies $a b \ind_{A} c$.
	\item Right extension (over $A$): if $a \ind_{A} b$ then for all $c$ there is $c' \equiv_{A b} c$ with $a \ind_{A} b c'$. 
		\item Left extension (over $A$): if $a \ind_{A} b$ then for all $c$ there is $c' \equiv_{A a} c$ with $a c' \ind_{A} b$. 
		\item Boundedness: there is a function $f$ on cardinals so that for every small set $C$ and $p(x) \in S(C)$ with $x$ a finite tuple, for every small $D \supseteq C$ the size of the set $\{\tp(a/D) : a \models p \land a \ind_{C} D\}$ is bounded by $f(|C|)$.
	\end{enumerate}
	
	We will say that $\ind$ is a \emph{standard preindependence relation} if it satisfies (1)--(6). 
\end{definition}

\begin{defn}
	As usual, a formula $\varphi(x,b) \in L(\cU)$ \emph{divides} over a (small) set $A$ if there exists an $A$-indiscernible sequence $(b_i)_{i \in \omega}$ so that $b_i \equiv_{A} b$ and $\{\varphi(x,b_i) : i \in \omega \}$ is inconsistent (equivalently, we may require $b_0 = b$ and $k$-inconsistency for some $k \in \omega$). And $\varphi(x,b)$ \emph{forks} over $A$ if $\varphi(x,b) \vdash \bigvee_{i < n} \psi_i(x,c_i)$ for some $n \in \omega$ and some formulas $\psi_i(x,c_i) \in L(\cU)$ dividing over $A$. A (possibly large) partial type $\pi(x)$ forks over $A$ if $\pi(x) \vdash \varphi(x,b)$ for some formula $\varphi(x,b)$ forking over $A$.
\end{defn}

\begin{remark}
	Equivalently, $\tp(a/Ab)$ does not divide over $A$ if and only if for any $A$-indiscernible sequence $I$ with $b \in I$ there is an $Aa$-indiscernible sequence $I'$ (or just all elements of $I'$ have the same type over $Aa$)  so that $I' \equiv_{Ab} I$.
\end{remark}

Let us recall Lascar strong types.
\begin{defn}
	Let $\Autf(\cU/A)$ be the subgroup of all automorphisms of $\cU$ generated by the set $\{f \in \Aut(\cU/M) : M \supseteq A \textrm{ is some small model}\}$. We write $a \equiv^{L}_{A} b$ (``$a$ and $b$ have the same Lascar strong type over $A$'') if $b = \sigma(a)$ for some $\sigma \in \Autf(\cU/A)$.
\end{defn}
\begin{fact}\label{fac: Lasc str type}
	The relation $\equiv^L_{A}$ is the finest $\Aut(\cU/A)$-invariant equivalence relation with boundedly many classes (on tuples of fixed length). It is also defined as the transitive closure of the relation $E(a,b)$ saying that there is an $A$-indiscernible sequence containing both $a$ and $b$ (see e.g.~\cite[Corollary 33]{adler2008introduction}).
\end{fact}

\begin{defn}
\begin{enumerate}
\item We write $a \ind^f_{A} b$ if $\tp(a/Ab)$ does not fork over $A$.
\item We write $a \ind^u_{A} b$ if $\tp(a/Ab)$ is finitely satisfiable in $A$, i.e.~every formula $\varphi(x) \in \tp(a/Ab)$ is satisfied by some tuple in $A$.
	\item We write $a \ind^i_{A} b$ if there is a global type $p$ extending $\tp(a/Ab)$ which is Lascar-invariant over $A$, i.e.~for every $c \equiv^{L}_{A} d$ and $\varphi(x,y) \in \mathcal{L}(A)$, $\varphi(x,c) \in p \Leftrightarrow \varphi(x,d) \in p$.
\end{enumerate}
	
\end{defn}

\begin{fact}\label{fac: indisc pres for lasc inv}
	(Any $T$.) If $I$ is an $A$-indiscernible sequence and $a \ind^i_{A} I$ then $I$ is indiscernible over $Aa$. And $a \ind^i_{A} b$ if and only if for any finitely many $A$-indiscernible sequences (of tuples) $I_1, \ldots, I_n$ there exist sequences $I'_1, \ldots, I'_n$ so that $I'_1 \ldots I'_n \equiv_{Ab} I_1 \ldots I_n$ and each $I'_i$ is indiscernible over $Aa$ (see \cite[Remark 2.20]{chernikov2012forking}).
\end{fact}

\begin{fact}(Any $T$)\label{fac: props of forking etc in NIP}
\begin{enumerate}
\item $\ind^f$ is a standard preindependence relation (see e.g.~\cite[Section 5]{adler2008introduction}).
	\item $\ind^u$ is a standard bounded preindependence relation and also satisfies left extension over models (see e.g.~\cite[Remark 2.16]{chernikov2012forking}).
	\item $\ind^i$ is a standard bounded (by $2^{2^{\kappa}}$) preindependence relation, and $\ind \Rightarrow \ind^i$ for any standard bounded preindependence relation $\ind$ (see \cite[Remark 2.20]{chernikov2012forking} and \cite[Corollary 35 and Proposition 37]{adler2008introduction}), 
	\item $a \ind^u_{A} b \Rightarrow a \ind^i_{A} b \Rightarrow a \ind^f_{A} b$ (see e.g.~\cite[Section 2]{chernikov2012forking}).
\end{enumerate}	
\end{fact}

\begin{fact}
 \cite[Corollary 38]{adler2008introduction}	The following are equivalent in any theory:
	\begin{enumerate}
		\item $\ind^f$ is bounded,
		\item $\ind^f$ is bounded by the function $f(\kappa) = 2^{2^{\kappa}}$,
		\item $\ind^f = \ind^i$.
	\end{enumerate}
\end{fact}

We recall that $\NTP_2$ is a larger class of theories containing all NIP and simple theories (see e.g.~\cite{NTP2}).
\begin{fact}\label{fac: forking in NIP equals inv}
\begin{enumerate}
	\item $T$ is NIP if and only if either of $\ind^u, \ind^i, \ind^f$ is bounded by $2^{\kappa}$; in particular, $\ind^i = \ind^f$ assuming NIP (see \cite{shelah2009dependent}, \cite[Corollary 38, Theorem 42]{adler2008introduction}).
	\item If $T$ is NTP$_2$ and $\ind^f$ is bounded, then $T$ is NIP (\cite[Theorem 4.3]{chernikov2012forking}; but this implication does not hold outside of $\NTP_2$ \cite{chernikov2016non}).
\end{enumerate}
\end{fact}

\begin{remark}
		\begin{enumerate}
	\item Assume $T$ is stable, $p$ is a global type, and $A$ a small set. Then $p$ does not fork over $A$ if and only if $p$ is $\acl^{\eq}(A)$-invariant (see e.g.~\cite{pillay1996geometric}).
		\item If $T$ is NIP, we have in fact a stronger version of $\ind^f=\ind^i$:  $p \in S_x(\cU)$ a global type and $A \subseteq \cU$ small, the $p$ does not fork over $A$ if and only if $p$ is $\Aut(\cU/\bdd(A))$-invariant, where $\bdd(A)$ denotes the bounded closure of $A$ in $\cU^{\heq}$ (\cite[Proposition 2.11]{hrushovski2011nip}).
	\end{enumerate}
\end{remark}

\subsection{Extension bases}

\begin{defn}
A small set $A \subseteq \mathbb{M}$ is an \emph{extension base} for $\ind$ if every complete type $p(x)$ over $A$ admits a global extension which is $\ind$-free over $A$. That is,  using right extension, if $a \ind_{A} A$ for all tuples $a$.
A theory $T$ is \emph{$\ind$-extensible} if every small set is an extension base.
\end{defn}

\begin{fact}\label{fac: ext bases stuff}
\begin{enumerate}
\item In any theory $T$, every model $M$ is an extension base for $\ind^u$ (and hence for $\ind^i$ and $\ind^f$).
\item If $T$ is stable (or just simple), then every set is an $\ind^f$-extension base. In particular, if $T$ is stable then every set is an $\ind^i$-extension base.
	\item If $T$ is a dp-minimal expansion of a linear order then all sets are $\ind^i$-extension bases \cite{simon2011dp}.
\end{enumerate}
\end{fact}


The following was shown with $3$ indiscernible sequences in \cite{yaacov2014independence} (in the sense of Fact \ref{fac: Lasc str type}), and improved to $2$ sequences in \cite{simon2020amalgamation}:
\begin{fact}
	If $T$ is $\NTP_2$ and $A$ is an $\ind^f$-extension base, then tuples $a,b$ satisfy $a \equiv^{L}_{A} b$ if and only if there exists $c$ so that both $a,c$ and $c,b$ start infinite indiscernible sequences over $A$. \end{fact}

Forking in NIP (or even $\NTP_2$) theories is well behaved over $\ind^f$-extension bases:
\begin{fact}\label{fac: forking = div NTP2}\cite[Theorem 1.2]{chernikov2012forking}
Assume $T$ is $\NTP_2$. Then for a small set $A$, the following are equivalent:
\begin{enumerate}
\item $A$ is an extension base for $\ind^f$; 
	\item $\ind^f$ satisfies left extension over $A$;
	\item  for any formula $\varphi(x,a) \in \mathcal{L}(\mathbb{M})$, it divides over $A$ if and only if it forks over $A$.
\end{enumerate}
In particular, assuming NIP, (1)$\Leftrightarrow$(2) for $\ind^i$.
\end{fact}

\subsection{Independent sequences}

In this section we let $\ind$ be a standard preindependence relation, unless specified further, and $I$ a linear order.

\begin{definition}
	Given $B \subseteq C$, we say that $\langle A_i : i \in I \rangle$ is an \emph{$\ind_{B}$-sequence over $C$} if $A_i \ind_{B} A_{<i} C$ for all $i \in I$.
\end{definition}

The following are some natural operations on $\ind$-sequences that we will use freely:
\begin{remark}\label{rem: grouping ind seq}
	Assume $B \subseteq C$.
	\begin{enumerate}
		\item If $\langle A_i : i \in I \rangle$ is a $\ind_{B}$-sequence over $C$ and $A'_i \subseteq A_i$ for all $i \in I$, then $\langle A'_i : i \in I \rangle$ is a $\ind_{B}$-sequence over $C$ (by monotonicity of $\ind$).
		\item If $\langle A_i : i \in I \rangle$ is a $\ind_{B}$-sequence over $C$ and $I_j$ is a convex subset of $I$ with $I_j < I_{j'}$ for all $j < j' \in J$, then $\langle A_{I_j} : j \in J \rangle$ is a $\ind_{B}$-sequence over $C$ (by finite character, base monotonicity and transitivity).
		\item Conversely, if $\langle A_\alpha : \alpha \in I \rangle$ is a $\ind_{B}$-sequence over $C$ and $A_\alpha = \bigcup_{i \in I_\alpha}A'_{\alpha,i}$ with $\langle A'_{\alpha,i} : i \in I_\alpha \rangle$ a $\ind_{B}$-sequence over $C A_{<\alpha}$ for all $\alpha \in I$, then  $\langle A'_{\alpha,i} : (\alpha,i) \in \sum_{\beta \in I}I_\beta  \rangle$ is a $\ind_{B}$-sequence over $C$, where $\sum_{\beta \in I}I_\beta$ is the lexicographic sum (i.e.~$\sum_{\beta \in I}I_\beta = \{(\alpha,i) : \alpha \in I \land i \in I_{\alpha}\}$ and $(\alpha_1,i_1) <_{\lex} (\alpha_2,i_2)$ if either $\alpha_1 < \alpha_2$ in $I$, or $\alpha_1 = \alpha_2$ and $i_1 < i_2$ in $I_{\alpha}$). This is again not hard to check using finite character, base monotonicity and transitivity.
	\end{enumerate}	
	\end{remark}

In the following lemma, (1) is an analog of  \cite[Part I Lemma
2.6]{shelah2006monadic} and \cite[Proposition 2.10]{braunfeld2021characterizations} and (2) is an analog of \cite[Lemma 2.20]{braunfeld2021characterizations}, but with a general pre-independence relation instead of $\ind^u$ so our proof is different.

\begin{lemma}\label{lem: ext base for ind seq}
\begin{enumerate}
	\item Suppose $C \supseteq B$, $\alpha$ is an ordinal and  $(A_i : i < \alpha)$ is an $\ind_{B}$-sequence over $C$. Then for every $D \supseteq C$ there exists $D' \equiv_{C} D$ so that $(A_i : i < \alpha)$ is an $\ind_{B}$-sequence over $D'$.
	\item Suppose $I$ is infinite and $\langle A_i : i \in I \rangle$ is an $\ind_{B}$-sequence over $B$ which is also $B$-indiscernible. Then for any $C \supseteq B$ there is $C' \equiv_{B} C$ so that $\langle A_i : i \in I \rangle$ is a $C'$-indiscernible $\ind_{B}$-sequence over $C'$.
\end{enumerate}
	
\end{lemma}
\begin{proof}

(1) By transfinite induction on $i < \alpha$ we choose $D_i$ so that $A_i \ind_{B} A_{<i} D_i$, $D_0 \equiv_{C} D$ and $D_i \equiv_{C A_{\leq j}} D_j$ for all $j \leq i$.

For $i=0$, as $A_0 \ind_{B} C$ by assumption, we find $D_0 \equiv_{C} D$ with $A_0 \ind_{B} D_0$ by right extension. 

Assume $0 < i = i' + 1 <\alpha$ is a successor ordinal. Let $D_{i'}$ be given by the inductive assumption. By assumption we have $A_{i} \ind_{B} A_{\leq i'} C$, so by right extension we find $D_i \equiv_{A_{\leq i'} C} D_{i'}$ with $A_i \ind_{B} A_{\leq i'} D_i$. By invariance and inductive assumption on $D_{i'}$ we still have $D_i \equiv_{A_{\leq j} C} D_{j}$ for all $j \leq i$.

Assume $0<i < \alpha$ is a limit ordinal. Let $D'_{i} \models \bigcup_{j < i} \tp(D_j/ C A_{\leq j})$ (by the inductive assumption we have $\tp(D_{j_1}/CA_{\leq {j_1}}) \subseteq \tp(D_{j_2}/CA_{\leq {j_2}})$ for all $j_1 \leq j_2 < i$, so the union is consistent). By assumption, $A_{i} \ind_{B} A_{<i} C$, so by right extension we find $D_i \equiv_{A_{<i} C} D'_i$ with $A_i \ind_{B} A_{<i} D_i$. By the choice, for any $j < i$ we have $D_i \equiv_{A_{\leq j} C} D'_i \equiv_{A_{\leq j} C} D_{j}$.

Having carried out the induction, if $\alpha = \alpha'+1$ is a successor then we already have $D_{\alpha'} \equiv_{C} D$ and  $A_i \ind_{B} A_{<i} D_{\alpha'}$ for all $i < \alpha$ (by invariance of $\ind$, as $A_i \ind_{B} A_{<i} D_i$ by induction and  $A_i A_{<i} B D_i \equiv A A_{<i} B D_{\alpha'}$).

If $\alpha$ is a limit, we let $D_{\infty} \models \bigcup_{i < \alpha} \tp(D_i/ C A_{\leq i})$ (consistent, as above; in particular $D_{\infty} \equiv_{C} D$). We claim that $A_i \ind_{B} A_{<i} D_{\infty}$ for all $i < \alpha$. Suppose that $A_i \nind_{B} A_{<i} D_{\infty}$ for some $i < \alpha$. As $D_{\infty} \equiv_{C A_{\leq i}} D_i$, by invariance this implies $A_i \nind_{B} A_{<i} D_{i}$, contradicting the choice of $D_i$.

(2) Let $\kappa$ be a cardinal sufficiently large with respect to $|C|$. As $I$ is infinite, by Ramsey and compactness, let $\langle A'_i : i < \kappa \rangle$ be a $B$-indiscernible sequence so that $\EM(\langle A_i : i \in I \rangle / B) \subseteq \EM(\langle A'_i : i < \kappa \rangle / B)$. In particular, using $B$-indiscernibility of $\langle A_i : i \in I \rangle$, for any $n < \omega$ and $i_0 < \ldots < i_{n-1} < \kappa$ we have $A'_{i_0} \ldots A'_{i_{n-1}} \equiv_{B} A_{j_0} \ldots A_{j_{n-1}} $ for some/any $j_0 < \ldots < j_{n-1} \in I$. Then, by invariance and finite character of $\ind$, we still have that $\langle A'_i : i < \kappa \rangle$ is an $\ind_{B}$-sequence over $B$. Applying part (1), we find some $C' \equiv_{B} C$ so that $\langle A'_i : i < \kappa \rangle$ is a $\ind_{B}$-sequence over $C'$. As $\kappa \gg |C'|$, by Erd\H{o}s-Rado and compactness we find a $C'$-indiscernible sequence $\langle A''_i : i \in I \rangle$ so that for every $n < \omega$ and $i_0 < \ldots < i_{n-1} \in I$ we have $A''_{i_0} \ldots A''_{i_{n-1}} \equiv_{C'} A'_{j_0} \ldots A'_{j_{n-1}} $ for some $j_0 < \ldots < j_{n-1} < \kappa$. By invariance and finite character of $\ind$ again, we still have that $\langle A''_i : i \in I \rangle$ is an $\ind_{B}$-sequence over $C'$. We still have $\EM(\langle A''_i : i \in I \rangle / B) \subseteq \EM(\langle A_i : i \in I  \rangle / B)$, so $\langle A''_i : i \in I \rangle \equiv_{B} \langle A_i : i \in I \rangle $ (by indiscernibility of $ \langle A_i : i \in I \rangle$ over $B$). Taking $C''$ so that $\langle A''_i : i \in I \rangle C' \equiv_{B} \langle A_i : i \in I \rangle C''$, we get that $C'' \equiv_{B} C$,  $\langle A_i : i \in I \rangle$ is $C''$-indiscernible and $\langle A_i : i \in I \rangle$ is an $\ind_{B}$-sequence over $C''$ (by invariance of $\ind$).
%
%
%
\end{proof}

The following is a standard argument for the existence of Morley sequences:
\begin{lemma}\label{lem: existence of ind MS}
	If $a \ind_{B} C$, then for any small linear order $I$ there exists an $\ind_{B}$-sequence $\langle a_i : i \in I \rangle$ over $C$ with $a_i \equiv_{C} a$ which is moreover $C$-indiscernible.
\end{lemma}
\begin{proof}
	Fix a cardinal $\kappa$ sufficiently large with respect to $|C|$. We can choose $\langle a_i : i < \kappa \rangle$ so that $a_0 = a$, $a_i \ind_{B} a_{<i} C$ and $a_i \equiv_{C a_{<j}} a_{j}$ for all $i < j < \kappa$. For $0 < i = i'+1 < \kappa$ a successor, as $a_{i'} \ind_{B} a_{<i'}C$ by inductive assumption, by right extension (and automorphism)  we find $a_i \equiv_{a_{<i'} C} a_{i'}$ with $a_i \ind_{B} a_{i'} a_{<i'} C$. For $0 <i<\kappa$ limit, we let $a_i \models \bigcup_{j < i} \tp(a_j/a_{<j}C)$. These types are increasing, hence the union is consistent. And if $a_i \nind_{B} a_{<i}$ then by finite character and monotonicity $a_i \nind_{B} a_{<j}$ for some $j < i$, hence $a_{j} \nind_{B} a_{<j}$ by invariance as $a_i \equiv_{C a_{<j}} a_j$, a contradiction.
	
	So $\langle a_i : i < \kappa \rangle$ is an $\ind_{B}$-sequence over $C$ with $a_i \equiv_{C} a$. By Erd\H{o}s-Rado and compactness we find a $C$-indiscernible sequence $\langle a'_i : i  \in I \rangle$ so that for every $n < \omega$ and $i_0 < \ldots < i_{n-1} \in I$ we have $a'_{i_0} \ldots a'_{i_{n-1}} \equiv_{C} a_{j_0} \ldots a_{j_{n-1}} $ for some $j_0 < \ldots < j_{n-1} < \kappa$. In particular $\langle a'_i : i  \in I \rangle$ is still an $\ind_{B}$-sequence over $C$ (by invariance and finite character of $\ind$) and $a'_i \equiv_{C} a_i$.
\end{proof}

\subsection{Full sets}
Again, we let $\ind$ be a standard preindependence relation.

\begin{definition}\label{def: full sets}
	We say that $C \supseteq A$ is \emph{$\ind$-full over $A$} if for any $D \supseteq C$ and any $a \equiv_{C} a'$ with $a \ind_{A} D$ and $a' \ind_{A} D$, we must have $a \equiv_{D} a'$ (i.e.~every type over $C$ admits a unique  global $\ind_{A}$-free extension).
\end{definition}

\begin{remark}\label{rem: full sets aut}
	Note that if $C \supseteq A$ is $\ind$-full over $A$ and $C' \equiv_{A} C$, then $C'$ is also $\ind$-full over $A$ (by invariance of $\ind$).\end{remark}

The following lemmas are straightforward generalizations of some lemmas in \cite{shelah2006monadic} from $\ind^u$ to arbitrary (and sometimes bounded) standard preindependence relations. See also the presentation in \cite{braunfeld2021characterizations}. The following lemma generalizes \cite[Part I Lemma 1.5]{shelah2006monadic}.
\begin{lemma}\label{lem: existence of full sets}
	If $\ind$ is bounded and $C \supseteq A$ contains a representative of every Lascar strong type (of a finite tuple) over $A$, then $C$ is $\ind$-full over $A$ (so, assuming NIP, $C$ is also $\ind^f$-full over $A$). In particular, this happens if there is some model $M \prec \cU$ with $A \subseteq M \subseteq C$ and every $p \in S_{<\omega}(M)$ is realized in $C$. 
\end{lemma}
\begin{proof}
	 Assume $D \supseteq C$ and $a \equiv_{C} a'$ with $a \ind^i_{A} D$ and $a' \ind^i_{A} D$ are so that $a \not \equiv_{D} a'$, say $\models \varphi(a, b) \land \neg \varphi(a',b)$ for some $\varphi(x,y) \in \Lcal$ and $b$ a finite tuple from $D$. Let $p,p'$ be global types Lascar-invariant over $A$ with $p \supseteq \tp(a/D), p' \supseteq \tp(a'/D)$, so $\varphi(x,b) \in p, \neg \varphi(x,b) \in p'$. By assumption on $C$, there is some tuple $b'$ in $C$ with $b' \equiv^L_{A} b$, so by $A$-Lascar invariance of $p,p'$ we still have $\varphi(x,b') \in p, \neg \varphi(x,b') \in p'$, so $\models \varphi(a,b') \land \neg \varphi(a',b')$, contradicting $a \equiv_{C} a'$.
\end{proof}

\begin{remark}
	It follows that a standard preindependence relation $\ind$ is bounded if and only if for every small set $A$ there is a small set $C \supseteq A$ which is $\ind$-full over $A$, in which case we can choose $C$ with $|C| \leq 2^{|A| + |T|}$.
\end{remark}

The following lemma generalizes \cite[Part I Observation 1.6]{shelah2006monadic}.
\begin{lemma}\label{lem: split element in ind-seq}
	If $D \supseteq C$ is $\ind$-full and $\langle A, (B_1, B_2) \rangle $ is an $\ind_{C}$-sequence over $D$, then $\langle A, B_1, B_2 \rangle$ is an $\ind_{C}$-sequence over $D$ if and only if $\langle B_1, B_2 \rangle$ is an $\ind_{C}$-sequence over $D$.
\end{lemma}
\begin{proof}
	Left to right is clear by monotonicity. For right to left, we only need to show $B_2 \ind_{C} A B_1 D $. As $B_2 \ind_{C} B_1 D$ by assumption, by  right extension (and automorphism) choose 
	$B'_2 \equiv_{B_1 D} B_2$ with $B'_2 \ind_{C} A B_1 D$, so $B'_2 \ind_{C B_1} A  D$ by base monotonicity. By assumption $B_1 B_2 \ind_{C} A D$ and $B_1 \ind_{C} A D$, so also $B_1 B'_2 \ind_{C} A D$ by the above and left transitivity. As we have $B_1 B_2 \equiv_{D} B_1 B'_2$ and $D$ is $\ind$-full over $C$, this implies $B_1 B'_2 \equiv_{A D} B_1 B_2$. As $B'_2 \ind _{C} A B_1 D$, also $B_2 \ind _{C} A B_1 D$ by invariance.
\end{proof}

The following lemma generalizes \cite[Part I Lemma 2.6]{shelah2006monadic}.
\begin{lemma}\label{lem: types in ind seqs}
Assume $D \supseteq C$ is $\ind$-full.
\begin{enumerate}
	\item If $B_i \ind_{C} A_i D$  for $i \in \{0,1\}$ and $A_0 \equiv_{D} A_1$, $B_0 \equiv_D B_1$, then also $A_0 B_0 \equiv_{D} A_1 B_1$.
	\item If $\langle A_i : i \in I \rangle$ is an $\ind_{C}$-sequence over $D$, then for any $n < \omega$ and $i_0 < \ldots < i_{n-1}$, $j_0 < \ldots < j_{n-1}$ in $I$, if $A_{i_t} \equiv_{D} A_{j_t}$ for all $t<n$, then $A_{i_0} \ldots A_{i_{n-1}} \equiv_{D A_{< i^{\ast}}} A_{j_0} \ldots A_{j_{n-1}}$, where $i^{\ast} := \min\{i_0, j_0\}$.
\end{enumerate}
	
\end{lemma}
\begin{proof}
	(1) As $A_0 \equiv_{D} A_1$, let $B'_1$ be so that $A_0 B_0 \equiv_{D} A_1 B'_1$, in particular $B'_1 \equiv_{D} B_0 \equiv_{D} B_1$. As $B_0 \ind_{C} A_0 D$, by invariance $B'_1 \ind_{C} A_1 D$, and also $B_1 \ind_{C} A_1 D$ by assumption. As $D$ is full over $C$, this implies $B'_1 \equiv_{A_1 D} B_1$. Hence $B_1 A_1 \equiv_{D} B'_1 A_1 \equiv_{D} B_0 A_0$.
	
	(2) Argue by induction on $n$. By assumption we have $A_{i_{n}} \ind_{C} A_{<i^{\ast}} A_{i_0} \ldots A_{i_{n-1}} D$, $A_{j_{n}} \ind_{C} A_{<i^{\ast}} A_{j_0} \ldots A_{j_{n-1}} D$ and $A_{i_n} \equiv_{D} A_{j_n}$, and  $A_{<i^{\ast}} A_{i_0} \ldots A_{i_{n-1}} \equiv_{D} A_{<i^{\ast}} A_{j_0} \ldots A_{j_{n-1}} $ by the inductive assumption, so we conclude by (1).
\end{proof}

\section{Triviality of forking}
\subsection{BS-triviality for bounded preindependence relations}

The following property was considered for forking in monadically stable theories in \cite{baldwin1985second}, and for finite satisfiability over models in monadically NIP theories in \cite{shelah2006monadic} (it is called ``the f.s.~dichotomy'' in \cite{braunfeld2021characterizations}).

\begin{defn}
	We say that $\ind$ satisfies the \emph{Baldwin-Shelah triviality}, or \emph{BS-triviality}, over $A$, if for any tuples $\bar{a}, \bar{b}$ with $\bar{b} \ind_{A} \bar{a}$ and any singleton $c$ we have that either $\bar{b} c \ind_{A} \bar{a}$ or $\bar{b} \ind_{A} \bar{a} c$.
\end{defn}

The notion of a theory ``admitting coding''
was central in \cite{baldwin1985second} and \cite{shelah2006monadic}, and its variants play an important role in \cite{simon2021ordered, braunfeld2021characterizations, braunfeld2025indiscernibles}. We will use the following: 
\begin{fact}\cite{braunfeld2021characterizations, braunfeld2024corrigenda}
	If $T$ is monadically NIP, there do not exist a $D$-indiscernible
sequence $(d_i:i\in\mathbb Q)$ of tuples, a singleton $e$, a formula
$\varphi(u,v,z)\in\Lcal(D)$ and $s<t$ such that $
\models \varphi(d_s,d_t,e)$, $
\models \neg\varphi(d_i,d_t,e)$ for all $i<s$, and $
\models \neg\varphi(d_s,d_j,e)$ for all $j>t$.
\end{fact}

The following is the key result for forking. The proof is based on \cite[II, Lemma 2.3]{shelah2006monadic}, and its presentation in \cite[Proposition 3.11(4) implies (1)]{braunfeld2021characterizations}, but with some variations. 
\begin{theorem}\label{thm: dich for ind rel mon NIP}
	Assume $T$ is monadically NIP, $\ind$ is a  bounded standard preindependence relation, $A$ is an extension base for $\ind$, and $\ind$ satisfies left extension over $A$. Let $\bar{a}, \bar{b}$ arbitrary tuples with $\bar{b} \ind_{A} \bar{a}$. Then $\ind$ is BS-trivial over $A$.
\end{theorem}
\begin{proof}
	Suppose $A$, $\bar{a}, \bar{b}, c$ are such that $\bar{b} \ind_{A} \bar{a}$, but $\bar{b}c \nind_{A} \bar{a}$ and $\bar{b} \nind_{A} \bar{a} c$.

	As $\ind$ is bounded, by Lemma \ref{lem: existence of full sets} there is a small set $D' \supseteq A$ so that $D'$ is $\ind$-full over $A$.
	Note that $\langle \bar{a}, \bar{b} \rangle$ is an $\ind_{A}$-sequence over $A$ (we have $\bar{a} \ind_{A} A$ as $A$ is an $\ind$-extension base). By Lemma \ref{lem: ext base for ind seq} (and Remark \ref{rem: full sets aut}) we can choose a small  $\ind$-full set $D \supseteq A$ so that $\langle \bar{a}, \bar{b} \rangle$ is an $\ind_{A}$-sequence over $D$. Then by base monotonicity and left transitivity we have in particular $\bar{a} \bar{b} \ind_{A} D$.

	By left extension over $A$ there is $c' \equiv_{A \bar{a} \bar{b}} c$ such that $\bar{a}\bar{b} c' \ind_{A} D$. By Lemma \ref{lem: existence of ind MS} we choose an $\ind_{A}$-sequence $\langle \bar{a}_i \bar{b}_i c_i : i \in \mathbb{Q} \rangle$ over $D$ which is $D$-indiscernible and with $\bar{a}_i \bar{b}_i c_i \equiv_{D} \bar{a} \bar{b} c'$. 
	
	Fix any $s<t \in \mathbb{Q}$. By invariance, we still have $\bar{b}_s c_s \nind_{A} \bar{a}_s$ and $\bar{b}_s \nind_{A} \bar{a}_s c_s$.
	
	By Remark \ref{rem: grouping ind seq}, $\langle \bar{a}_{<s}, (\bar{a}_s, \bar{b}_s), \bar{b}_t, \bar{b}_{>t} \rangle$ is an $\ind_{A}$-sequence over $D$ (of length $4$).
	In particular $\langle \bar{a}_{<s}, (\bar{a}_s, \bar{b}_s) \rangle$ is an $\ind_{A}$-sequence over $D$,  $\bar{b}_s \ind_{A} \bar{a}_s D$ (by the choice of $D$ and invariance) and $D$ is $\ind$-full over $A$, so by Lemma \ref{lem: split element in ind-seq} we have $\langle \bar{a}_{<s}, \bar{a}_s, \bar{b}_s \rangle$ is an $\ind_{A}$-sequence over $D$. So, by Remark \ref{rem: grouping ind seq}(3), $\langle \bar{a}_{<s}, \bar{a}_s, \bar{b}_s, \bar{b}_t, \bar{b}_{>t} \rangle$ is an $\ind_{A}$-sequence over $D$.
	
	Then, as $\bar{b}_s \equiv_{D} \bar{b}_t$ by indiscernibility, by Lemma \ref{lem: types in ind seqs} we get $\bar{b}_s \equiv_{D \bar{a}_{<s} \bar{a}_s \bar{b}_{>t}} \bar{b}_t$. Take $c^{\ast}$ so that $\bar{b}_s c_s \equiv_{D \bar{a}_{<s} \bar{a}_s \bar{b}_{>t}} \bar{b}_t c^{\ast}$.
	
	Then, as $\bar{b}_s \nind_{A} \bar{a}_s c_s$,  also $\bar{b}_t \nind_{A} \bar{a}_s c^{\ast }$ (by assumption, monotonicity and invariance). But for any $j > t$ we have $\bar{b}_j \ind_{A} \bar{a}_s c_s D$ (as $\langle \bar{a}_i \bar{b}_i c_i : i \in \mathbb{Q} \rangle$ is a  $\ind_{A}$-sequence over $D$, by monotonicity), hence also $\bar{b}_j \ind_{A} \bar{a}_s c^{\ast} D$ (as in particular $ c_s \equiv_{D  \bar{a}_s \bar{b}_{j}}  c^{\ast}$, by invariance).

	
	And as $\bar{b}_s c_s \nind_{A} \bar{a}_s$, also $\bar{b}_t c^{\ast} \nind_{A} \bar{a}_s$ (by assumption and invariance). But for any $i < s$
	we have $\bar{b}_s c_s \ind_{A}  \bar{a}_i D$ (as $\langle \bar{a}_i \bar{b}_i c_i : i \in \mathbb{Q} \rangle$ is a  $\ind_{A}$-sequence over $D$, by monotonicity), hence also $\bar{b}_t c^{\ast} \ind_{A} \bar{a}_i D$ (as in particular $ \bar{b}_s c_s \equiv_{D  \bar{a}_i} \bar{b}_t c^{\ast}$, by invariance).
	
%
%

	By strong finite character of $\ind$, there exist some formulas $\rho_1(\bar{x}, \bar{y},z), \rho_2(\bar{x}, \bar{y},z) \in \tp(\bar{a}_s, \bar{b}_t, c^{\ast}/A)$ so that:
	\begin{itemize}
		\item for any $\bar{b}^{+}$, if $\models \rho_1(\bar{a}_s, \bar{b}^{+}, c^{\ast})$, then  $\bar{b}^{+} \nind_{A} \bar{a}_s c^{\ast}$;
		\item for any $\bar{b}^+$ and $c^+$, if $\models \rho_2(\bar{a}_s, \bar{b}^{+}, c^{+})$  then $\bar{b}^+ c^+ \nind_{A} \bar{a}_s$; hence, by indiscernibility and invariance of $\ind$,  for any $i \in \mathbb{Q}$ we have: if $\models \rho_2(\bar{a}_i, \bar{b}^{+}, c^{+})$  then $\bar{b}^+ c^+ \nind_{A} \bar{a}_i$.
	\end{itemize} 
	
	Letting $\bar{d}_i := \bar{a}_i \bar{b}_i$, $\bar{w}_i := \bar{x}_i \bar{y}_i$ for $i \in \{1,2\}$ and $\varphi(\bar{w}_1, \bar{w}_2, z) := \rho_1( \bar{x}_1 , \bar{y}_2, z) \land \rho_2(\bar{x}_1 , \bar{y}_2, z) \in \Lcal(D)$, we have that 
	\begin{itemize}
		\item $\langle \bar{d}_i : i \in \mathbb{Q} \rangle$ is $D$-indiscernible,
		\item  $\models \varphi(\bar{d}_s, \bar{d}_t, c^{\ast})$,
		\item  $\models \neg \varphi(\bar{d}_s, \bar{d}_j, c^{\ast})$ for all $j > t$ (as $\bar{b}_j \ind_{A} \bar{a}_s c^{\ast}$, so necessarily $\models \neg  \rho_1(\bar{a}_s , \bar{b}_j, c^{\ast})$),
		\item $\models \neg \varphi(\bar{d}_i, \bar{d}_t, c^{\ast})$ for all $i<s$ (as $\bar{b}_t c^{\ast} \ind_{A} \bar{a}_i$, so necessarily $\models \neg  \rho_2(\bar{a}_i , \bar{b}_t, c^{\ast})$). \qedhere
	\end{itemize}
\end{proof}

This immediately gives the following: 
\begin{cor}\label{cor: mon NIP BS-triv fork}
	\begin{enumerate}
		\item If $T$ is monadically NIP and $A$ is an $\ind^f$-extension base (e.g.~$A \prec \cU$ is a model), then $\ind^f$ satisfies BS-triviality over $A$. 
		\item \cite{baldwin1985second} If $T$ is monadically stable, then $\ind^f$ satisfies BS-triviality over arbitrary sets.
				\item  \cite{shelah2006monadic} If $T$ is monadically NIP, then $\ind^u$ satisfies BS-triviality over models.  
	\end{enumerate}
\end{cor}
\begin{proof}
	Combining Theorem \ref{thm: dich for ind rel mon NIP} with Facts \ref{fac: props of forking etc in NIP}, \ref{fac: forking in NIP equals inv} and \ref{fac: ext bases stuff}(2).
\end{proof}

\subsection{BS-triviality for forking over arbitrary sets}
 
The proof of Theorem~\ref{thm: dich for ind rel mon NIP} applies to a general bounded
preindependence relation, working over an extension base. While we show in Section \ref{sec: Trees with circularly ordered open cones}
 that there exist monadically NIP theories in which no finite set is an $\ind^f$-extension base, we can give an alternative argument specifically  for forking that has the advantage of working over arbitrary sets.


\begin{lemma}\label{arb:mutual-morley}
Let $A\subseteq M\preccurlyeq\cU$, let $p(y)$ be a global type
Lascar-invariant over $A$, and let $J\subseteq M$ be an endless
$A$-indiscernible sequence. If $I$ is a Morley sequence of $p$ over
$M$, then $I,J$ are mutually indiscernible over $A$.
\end{lemma}
\begin{proof}
Since $\Aut(\cU/M)\subseteq\Autf(\cU/A)$, the type $p$ is
$M$-invariant. Thus $I$ is indiscernible over $M$, and hence over
$AJ$. For the other direction, let $(b_1,\ldots,b_n)$ be any
finite increasing tuple from $I$. We prove by induction on $n$ that
$J$ is indiscernible over $Ab_1\ldots b_n$ (which is sufficient).

Suppose this  holds for $\bar b=(b_1,\ldots,b_{n-1})$, with
$\bar b$ empty at the first step. As $J$ is endless and $A \bar b$-indiscernible, two increasing tuples $\bar e,\bar e'$
of the same length from $J$ have the same Lascar strong type over
$A\bar b$. Thus  $
 (\bar b,\bar e)\equiv_A^L(\bar b,\bar e')$. 
For every formula $\rho(y;\bar v,\bar w)\in L(A)$,
Lascar invariance of $p$ over $A$ then gives $
 \rho(y;\bar b,\bar e)\in p \ \Leftrightarrow \ 
 \rho(y;\bar b,\bar e')\in p$. 
Both parameter tuples belong to $M\bar b$, and
$b_n\models p|_{M\bar b}$, so $
 \models \rho(b_n;\bar b,\bar e) \ \Leftrightarrow \ 
 \models \rho(b_n;\bar b,\bar e')$. This proves the induction step. 
 \end{proof}

\begin{theorem}\label{arb:forking-bs}
If $T$ is monadically NIP, then  $\ind^f$ satisfies BS-triviality 
over every small set $A$. 
\end{theorem}
\begin{proof}
Suppose, towards a contradiction, that $
 \bar b\ind^f_A \bar a$ but $\bar b c\nind^f_A \bar a$ and $\bar b\nind^f_A \bar a c$, where $c$ is a singleton. All formulas below are allowed parameters in $A$.

Choose $\theta(y,z; \bar a)\in\tp(\bar b c/A \bar a)$ which forks over $A$. There are finitely many formulas
$\psi_\ell(y,z;\bar e_\ell)$ with $\bar e_{\ell}$ in $\cU$, each dividing over $A$, such that 
\begin{equation}\label{arb:dividing-cover}
 \theta(y,z;\bar a)\ \longrightarrow\
 \bigvee_{\ell<m}\psi_\ell(y,z; \bar e_\ell).
\end{equation}
For every $\ell<m$, choose an $A$-indiscernible sequence $
 J_\ell=(\bar e_{\ell,t}:t\in\Q)$, with $\bar e_{\ell,0}= \bar e_\ell$, 
such that $\{\psi_\ell(y,z; \bar e_{\ell,t}):t\in\Q\}$ is inconsistent.
Let $M$ be a small model containing $A \bar a$ and all the sequences $J_\ell$.

As $
 \bar b\ind^f_A \bar a$, let $p(y)$ be a global  extension of
$\tp(\bar b/A \bar a)$ non-forking over $A$, and let  $\bar b_0\models p|_M$. Choose $c_0$ with  $
 \bar b_0c_0\equiv_{A \bar a} \bar bc$. We still have $\models \theta(\bar b_0,c_0;\bar a)$ and $\bar b_0\nind^f_A \bar a c_0$. By \eqref{arb:dividing-cover}, for some $\ell<m$ we have 
\begin{equation}\label{arb:selected-disjunct}
\models  \psi_\ell(\bar b_0,c_0; \bar e_{\ell,0}).
\end{equation}
Fix that $\ell$ and let $J := J_\ell$.

Choose $\varphi(y;\bar a,c_0)\in\tp(\bar b_0/A \bar a c_0)$ which forks over $A$,
and let $\bar b_1\models p|_{M \bar b_0 c_0}$. Since $p$ does not fork over
$A$, it contains $\neg\varphi(y;\bar a,c_0)$. Thus
\begin{equation}\label{arb:two-truth-values}
 \models \varphi(\bar b_0; \bar a,c_0) \land \neg\varphi(\bar b_1; \bar a,c_0).
\end{equation}

We can extend the pair $(\bar b_0, \bar b_1)$ 
to a Morley sequence $I=(\bar b_t:t\in\Q)$ of $p$ over $M$ (extend $(\bar b_0, \bar b_1)$ to an infinite Morley sequence $(\bar b_t : t \in \omega)$ in $p$ over $M$; any increasing
pair in that sequence has the type of $(\bar b_0, \bar b_1)$ over $M$ and can
be carried to this pair by an automorphism over $M$, so we can find the desired sequence realizing $p^{\otimes I}$ with $(\bar b_0, \bar b_1)$ in the corresponding positions by compactness). The type $p$ is Lascar-invariant over
$A$ (by Fact \ref{fac: forking in NIP equals inv}(1)), hence $I,J$ are mutually
indiscernible over $A$  by Lemma~\ref{arb:mutual-morley}.

The sequence $I$ is indiscernible over $A \bar a$, because $A \bar a\subseteq M$,
but \eqref{arb:two-truth-values} shows that it is not indiscernible
over $A \bar a c_0$. As monadically NIP theories satisfy endless indiscernible triviality \cite[Theorem~2.13]{braunfeld2025indiscernibles}, $I$ is not 
indiscernible over $Ac_0$.

Also $J$ is indiscernible over $A \bar b_0$ (as $J$ is indiscernible over $A I$). As $\{\psi_\ell(y,z; \bar e_{\ell,t}):t\in\Q\}$ is inconsistent, 
$\models \neg\psi_\ell(\bar b_0,c_0; \bar e_{\ell,t})$ for some $t\in\Q$. But also  $\models  \psi_\ell(\bar b_0,c_0; \bar e_{\ell,0})$ by \eqref{arb:selected-disjunct}. Thus $J$ is not indiscernible over $A \bar b_0c_0$. By endless indiscernible triviality again,  $J$ is not indiscernible
over $Ac_0$. 

Thus 
$I,J$ are mutually indiscernible over $A$, $c_0$ is a singleton and neither of $I,J$ is indiscernible over $A c_0$ --- contradicting dp-minimality of monadically NIP theories \cite{blumensath2011simple, blumensath2011simplea}.
\end{proof}

Finally for this section, we observe that BS-triviality of forking also gives a characterization of monadic NIP.
\begin{proposition}\label{conv:fu}
For any $T$, the following are equivalent:
\begin{enumerate}
\item $T$ is monadically NIP.
\item for every $M\models T$, finite tuples $a,b$, and singleton $c$, if $
 b\ind^u_M a$ then either $
 bc\ind^i_M a$ or $\ b\ind^i_M ac$.
\item $\ind^u$ satisfies BS-triviality over models. 
\end{enumerate}
\end{proposition}
\begin{proof}
(1) implies (3) by \cite{shelah2006monadic}, and (3) implies (1) by \cite{braunfeld2021characterizations,braunfeld2024corrigenda}. We have (3) implies (2) as 
$\ind^u\Rightarrow\ind^i$. So assume (2), and fix $b\ind^u_M a$ and $c$. Let a small set 
$D_0\supseteq M$ be $\ind^i$-full over $M$ (Lemma \ref{lem: existence of full sets}), replacing it by a conjugate  $D$ over $M$ we may assume  $
 a\ind^u_M D$, $b\ind^u_M aD$. Indeed, we can arrange the first condition by right extension, and then
move $D$ over $Ma$ using $b\ind^u_M a$ to arrange the second.
By base monotonicity and left transitivity, $ab\ind^u_M D$.
Left extension gives $c'\equiv_{Mab}c$ with
$abc'\ind^u_M D$. Applying an automorphism over $Mab$ sending $c'$
to $c$, and replacing $D$ by its image, we obtain $
 b\ind^u_M aD$ and $abc\ind^u_M D$. 
 By (2)  we obtain either $
 bc\ind^i_M aD$ or  $b\ind^i_M acD$.  Both $\tp(bc/D)$ and $\tp(b/D)$ are finitely satisfiable in $M$, hence admit finitely satisfiable extensions to any bigger sets of parameters; as $\ind^u$ implies $\ind^i$ and the set $D$ is still $\ind^i$-full, we get that either  $bc\ind^u_M aD$ or $b\ind^u_M acD$. Monotonicity gives (3).
\end{proof}

\begin{cor}\label{conv:forking-characterization}
Suppose $T$ is NIP. The following are equivalent:
\begin{enumerate}
\item $T$ is monadically NIP;
\item $\ind^f$ is BS-trivial over every small set;
\item $\ind^f$ is BS-trivial over every $\ind^f$-extension base;
\item $\ind^f$  is BS-trivial over every model.
\end{enumerate}
\end{cor}
\begin{proof}
Theorem~\ref{arb:forking-bs} gives (1)$\Rightarrow$(2), and
(2)$\Rightarrow$(3)$\Rightarrow$(4) is immediate. For (4)$\Rightarrow$(1), assume we have tuples 
$b\ind^u_M a$, and let $c$ be a singleton. Then $b\ind^f_M a$,
so (3) gives $bc\ind^f_M a$ or $b\ind^f_M ac$. By Fact \ref{fac: forking in NIP equals inv}(1), this gives either $bc\ind^i_M a$ or $b\ind^i_M ac$. Hence $T$ is monadically NIP by Proposition \ref{conv:fu}.
\end{proof}

\subsection{Left total triviality and weak binarity}

In \cite{goode1991some}, Poizat considers \emph{total triviality} of forking in stable theories. In \cite{chernikov2026n}, we considered a version of this property beyond stability:
\begin{defn}
	We say that $\ind$ satisfies the \emph{left total triviality} over $A$, if for any tuples $\bar{a},\bar{b}, \bar{c}$, if $\bar{a} \ind_{A} \bar{c}$ and $\bar{b} \ind_{A} \bar{c}$, then $\bar{a} \bar{b} \ind_{A} \bar{c}$.
\end{defn}

\begin{cor}\label{cor: bs-left-total}
	Every standard preindependence relation satisfying BS-triviality over $A$ satisfies left total triviality over $A$.
\end{cor}

\begin{proof}

By finite character, we may assume that $\bar b$ is finite. We argue by induction on $|\bar{b}|$. Assume $\bar{a} \ind_{A} \bar{c}$, $\bar{b} \ind_{A} \bar{c}$ and $\bar{b} = (b_0, \ldots, b_{n})$ with $b_i$ singletons.
  By monotonicity in particular $b_0 \ldots b_{n-1} \ind_{A} \bar{c}$, hence by the inductive assumption $\bar{a} b_0 \ldots b_{n-1} \ind_{A} \bar{c}$. 
  By BS-triviality over $A$ we then have that either $\bar{a}  b_0 \ldots b_{n-1} b_n \ind_{A} \bar{c}$ and we are done, or $\bar{a}  b_0 \ldots b_{n-1} \ind_{A} \bar{c}  b_n$, hence $\bar{a}  b_0 \ldots b_{n-1} \ind_{A b_n} \bar{c}$ by base monotonicity. As also $b_n \ind_{A} \bar{c}$ by assumption and monotonicity, by left transitivity we again get $\bar{a}  b_0 \ldots b_{n-1} b_n \ind_{A} \bar{c}$.
\end{proof}

\begin{cor}\label{cor: arbitrary-left-total}
	If $T$ is monadically NIP and $A$ is any small set, then $\ind^f$ satisfies left total triviality over $A$.
\end{cor}
\begin{proof}
By Theorem~\ref{arb:forking-bs} and Corollary~\ref{cor: bs-left-total}.
\end{proof}

This quickly implies some results in the literature:
\begin{defn}\cite{mennuni2020product} 
A theory $T$ is \emph{weakly binary} if for any tuples $a,b$ (in a bigger monster model $\cU' \succ \cU$) so that $\tp(a/\cU)$ and  $\tp(b/\cU)$ are invariant (over some small subsets of $\cU)$, there is a small set $A \subseteq \cU$ so that $\tp(a/\cU) \cup \tp(b/\cU) \cup \tp(a,b/A) \vdash \tp(a,b/\cU)$.	
\end{defn}

\begin{proposition}
	Every NIP theory with left totally trivial forking is weakly binary. In particular, every monadically NIP theory is weakly binary.
\end{proposition}
\begin{proof}
	Let $M \prec \cU$ be small so that $\tp(a/\cU)$ and  $\tp(b/\cU)$ are $M$-invariant. Let $M \prec N \prec \cU$ be small and $|M|^{+}$-saturated, in particular $N$ is $M$-full (Definition \ref{def: full sets}). As $a \ind^i_{M}\cU$ and $b \ind^i_{M} \cU$,  by left total triviality of $\ind^i = \ind^f$ we get $a b \ind^i_{M} \cU$.

	Let $a',b'$ be arbitrary so that $a' \equiv_{\cU} a$, $b' \equiv_{\cU} b$ and $a' b' \equiv_{N} a b$.  In particular $a' \ind^{i}_{M} \cU$ and $b' \ind^{i}_{M} \cU$, hence by left total triviality  $a' b' \ind^i_{M} \cU$. But as $N$ is $\ind^i$-full over $M$ and $a b \equiv_{N} a' b'$, we get $ab \equiv_{\cU} a' b'$.
\end{proof}

\begin{cor}
	\cite{mennuni2022weakly} The theory of dense meet trees is weakly binary.
\end{cor}

We note that \emph{right} total triviality need not hold over models in monadically NIP theories (for $\ind^u$ or for $\ind^f$):
\begin{example}
	Let $T = \DLO$. Let $M$ be an arbitrary small model, $M = C \sqcup D$ for some cut $C < D$ with infinite cofinality on both sides, $a, b_1, b_2$ arbitrary with $C < b_1 < a < b_2 < D$. Then $a \ind^u_{M} b_1, a \ind^u_{M} b_2$ (and $b_1 \ind^u_{M} b_2, b_2 \ind^u_{M} b_1$) but $a \nind^f_{M} b_1 b_2$ (since the formula $b_1 < x < b_2$ divides over $M$). 
\end{example}
\begin{remark}
	However, something weaker holds:  if $\ind$ is BS-trivial over $A$, $C \supseteq A$ is $\ind$-full over $A$, and $\bar b$ is nonempty, then $\bar{a} \ind_{A} \bar{b} C$ if and only if $\bar{a} \ind_{A} b C$ for all singletons $b \in \bar{b}$ (this easily follows from BS-triviality in the same way as in \cite[Lemma 3.3(3)]{braunfeld2021characterizations} for $\ind^u$).
\end{remark}

\section{Strict non-forking and strict independence}\label{sec: strict nf}

\begin{definition}
	We say that $\tp(a/Ab)$ \emph{strictly does not fork} over $A$, and write $a \ind^{\st}_{A} b$, if there is a global type $p$ extending $\tp(a/Ab)$ which does not fork over $A$ (so in particular $a \ind^f_{A} b$) and for any $B \supseteq Ab$, if $c \models p|_{B}$ then $\tp(B/Ac)$ does not fork over $A$ (so $B \ind^f_{A} c$).
\end{definition}

Strict non-forking was introduced by Shelah \cite{shelah2009dependent}, and studied further in \cite{chernikov2012forking, chernikov2014theories, kaplan2014strict}.

\begin{definition}
\begin{enumerate}
	\item Given a set $S$, the sequences $(I_{i} : i \in S)$ are \emph{mutually indiscernible} over $A$ if $I_i$ is indiscernible over $I_{\neq i} A$ for all $i \in S$.
	\item Given a linear order $S$, the sequences $(I_{i} : i \in S)$ are \emph{almost mutually indiscernible} over $A$ if  there exists $a_i \in I_i$ so that $I_i$ is indiscernible over $I_{<i} a_{>i} A$ for all $i \in S$.
\end{enumerate}
\end{definition}

\begin{fact}\cite[Lemma 1.3]{chernikov2014theories} \label{fac: almost mut ind rotates mut ind}(Any $T$)
	Assume $(I_i : i \in S)$ are almost mutually indiscernible over $A$, witnessed by $a_i \in I_i$. Then there exist sequences $(I'_i : i \in S)$ mutually indiscernible over $A$ and with $I'_i \equiv_{A a_i} I_i$ for all $i \in S$.
\end{fact}

\begin{definition}
	A set of tuples $\{b_i : i \in S \}$ is \emph{strictly independent} over $A$ if for any $A$-indiscernible sequences $I_i$ so that $b_i$ is an element of $I_i$ for all $i \in S$, there exist some sequences $I'_i \equiv_{A b_i} I_i$ so that $(I'_i : i \in S)$ are \emph{mutually indiscernible} over $A$.
\end{definition}

\begin{fact}\cite[Theorem 6.8]{kaplan2014strict}
	Let $T$ be NIP and $A$ an $\ind^f$-extension base. The following are equivalent:
	\begin{enumerate}
		\item $a \ind^{\st}_{A} b$,
		\item $\{a,b\}$ is strictly independent over $A$.
	\end{enumerate}
	In particular, $a \ind^{\st}_{A} b \Leftrightarrow b \ind^{\st}_{A} a$.
\end{fact}

\begin{proposition}\label{prop: strict indep vs pairwise}
	Assume $T$ is NTP$_2$, $A$ is an $\ind^f$-extension base and $\ind^f$ is left totally trivial over $A$. Then the following are equivalent for any set of tuples $\{a_i : i \in S \}$:
	\begin{enumerate}
		\item $\{a_i : i \in S \}$ is strictly independent over $A$;
		\item $a_i \ind^f_{A} a_j$ for all $i \neq j \in S$.
	\end{enumerate}
\end{proposition}
\begin{proof}
	(1)$\Rightarrow$(2). Suppose $a_i \nind^f_{A} a_j$ for some $i \neq j \in S$, hence there is some formula $\varphi(x,y) \in \Lcal(A)$ so that $\models \varphi(a_i, a_j)$ and $\varphi(x, a_j)$ divides over $A$. Let $I_j \ni a_j$ be an $A$-indiscernible sequence witnessing this (i.e.~$\{\varphi(x,a') : a' \in I_j\}$ is inconsistent). Then no $I'_j \equiv_{A a_j} I_j$ can be $a_i$-indiscernible.

	(2)$\Rightarrow$(1). By compactness, it suffices to show this for finite $S$, so let $\{a_i : 0 \leq i \leq n\}$ with $a_i \ind^f_{A} a_j$ for all $i \neq j$ and $I_i \ni a_i$ indiscernible over $A$ be given. 
	
	By induction on $i$ we choose sequences $(I'_i : i < n)$ so that $I'_i \equiv_{A a_i} I_i$ and $I'_i$ is indiscernible over $A I'_{<i}a_{>i}$ for all $i < n$ (which is sufficient by Fact \ref{fac: almost mut ind rotates mut ind}).
	
	Let $0 \leq i < n$ and assume we have already chosen $(I'_j : 0 \leq j < i)$.
	
%
%
	
	By assumption and left total triviality  we have $a_{>i} \ind^{f}_{A} a_i$. Given $j < i$, we have $a_j \ind^f_{A} a_i$ by assumption. As $I'_j$ is $A a_i$-indiscernible by the inductive assumption, by invariance we have $a' \ind^{f}_{A} a_i$ for all elements $a' \in I'_{j}$, so $I'_j \ind^f_{A} a_i$ by left total triviality. Then, by left total triviality again, we have $I'_{<i} a_{>i} \ind^{f}_{A} a_i$. 
	 Hence there exists $I'_i \equiv_{Aa_i} I_i$ so that $I'_i$ is $AI'_{<i}a_{>i}$-indiscernible.
\end{proof}

In particular, \cite[Problems 10.1 and 10.2]{kaplan2014strict} have positive answer for NIP theories with left totally trivial forking (in particular for monadically NIP theories):
\begin{cor}\label{cor: strict f iff both dirs}
	Let $T$ be NIP, $A$ an $\ind^f$-extension base, and $\ind^{f}$ is left totally trivial over $A$. Then the following are equivalent:
	\begin{enumerate}
		\item $a \ind^{\st}_{A} b$,
		\item $a \ind^f_{A} b$ and $b \ind^f_{A} a$,
		\item the set $\{a,b\}$ is strictly independent over $A$.
	\end{enumerate}
\end{cor}

\begin{cor}
	Assume $T$ is NIP, $A$ is an $\ind^f$-extension base, and $\ind^{f}$ is left totally trivial over $A$. Then the following are equivalent for an $A$-indiscernible sequence $\langle a_i : i \in I\rangle$:
	\begin{enumerate}
		\item $\langle a_i : i \in I\rangle$  is a \emph{witness} over $A$ (i.e.~whenever $\varphi(x,a_i)\in L(Aa_i)$ divides
over $A$, the set $\{\varphi(x,a_j):j\in J\}$ is inconsistent),
		\item $a_{\neq i} \ind^{f}_{A} a_i$ for all $i \in I$,
		\item $a_i \ind^f_{A} a_j$ for all $i \neq j \in I$,
		\item $\{a_i : i \in I\}$ is strictly independent over $A$.
	\end{enumerate}
	 \end{cor}
\begin{proof}
	The equivalence $(1)\Leftrightarrow(2)$ is
\cite[Theorem 5.2]{kaplan2014strict} (NIP theories are resilient by \cite{yaacov2014independence}).
Monotonicity gives $(2)\Rightarrow(3)$, and left total triviality gives the converse. Proposition~\ref{prop: strict indep vs pairwise}
gives $(3)\Leftrightarrow(4)$.
\end{proof}

For the following consequence of the characterization of dp-rank, see \cite{chernikov2014theories} or \cite{kaplan2014strict}:
\begin{fact}\label{fac: dp-rk vis st}
Let $T$ be NIP and $A$ an $\ind^{f}$-extension base. If $\dprk(a/A) < n$, then for any set of $n$ finite tuples $\{b_i : i < n\}$ strictly independent over $A$, we have $a \ind^{f}_{A} b_i$ for some $i<n$.
\end{fact}
\noindent Combining Fact \ref{fac: dp-rk vis st} with Corollary \ref{cor: strict f iff both dirs} and the fact that monadically NIP theories are dp-minimal \cite{blumensath2011simple,  blumensath2011simplea}, we thus note the following immediate corollary which motivates the next section:
\begin{cor}\label{cor: key semilin fork}
	Let $T$ be monadically NIP and $A$ an $\ind^f$-extension base. Let $\bar{a}, \bar{b}$ be any finite tuples with $\bar{a} \ind^f_{A} \bar{b}$ and $\bar{b} \ind^{f}_{A} \bar{a}$. Then for any singleton $c$, either $c \ind^f_{A} \bar{a}$ or $c \ind^{f}_{A} \bar{b}$.
\end{cor}
\begin{example}
\begin{enumerate}
	\item We need to require both $\bar{a} \ind^f_{A} \bar{b}$ and $\bar{b} \ind^{f}_{A} \bar{a}$ here. Indeed, in $T = \DLO$, let $a_1 < b_1 < c < b_2 < a_2$. Then $\emptyset$ is an $\ind^f$-extension base (all sets are) and $a_1 a_2 \ind^{f} b_1 b_2$ (but $b_1 b_2 \nind^{f} a_1 a_2$), $c \nind^f a_1 a_2$ and $c \nind^{f} b_1 b_2$.
	\item It is necessary to assume that $A$ is an $\ind^f$-extension base. Consider the theory of two disjoint circular orders, let $a$ be a single point in one of them, and let $b, c$ be points in the same circle. Then $b \ind^f_{a} c$ and $c \ind^f_a b$. But for any point $e$ in the other circle, $e \nind^f_a b, e \nind^f_a c$ (because $e \nind^f_a a$, witnessed by the formula over $a$ saying ``I am in the other circle'').
\end{enumerate}
\end{example}

\section{The forking forest}\label{sec: fork forest}

In this section we consider a common generalization of the fundamental equivalence relation in monadically stable theories defined in terms of forking by Baldwin and Shelah \cite{baldwin1985second} and Shelah's partial order defined via $\ind^u$ over models in monadically NIP theories \cite{shelah2006monadic} (studied further in \cite{blumensath2011simple}). As the main result, we show that this partial order is \emph{semilinear}, after possibly removing one greatest equivalence class (i.e.~the set of predecessors of any element is linearly ordered).

\subsection{Semilinearity}

\begin{remark}\label{rem:invariant-transfer}
Suppose
$C\subseteq D\subseteq M \prec \cU$, $ u\equiv_M v$ and $e\ind^i_C Duv$. Then $u\equiv_{De}v$.
\end{remark}
\begin{proof}
By definition $\tp(e/Duv)$ extends to a global type
$p(x)$  Lascar-invariant over $C$. Choose
$\sigma\in\Aut(\cU/M)$ with $\sigma(u)=v$, then $\sigma \in \Autf(\cU/C)$ and 
it fixes $D$ pointwise. So $\varphi(x;d,u)\in p \Leftrightarrow 
 \varphi(x;d,v)\in p$ for any tuple $d$
from $D$ and formula $\varphi$.  We conclude as $e \models p|Duv$.\end{proof}

We have the following  generalization of Corollary \ref{cor: key semilin fork}:
\begin{theorem}\label{strict:relative-semilinearity}
Let $T$ be arbitrary, and let $\ind$ be a standard bounded 
preindependence relation satisfying BS-triviality over $C$ (in a monadically NIP theory, this applies to $\ind=\ind^f$ over every small set $C$, by Theorem~\ref{arb:forking-bs}, and to $\ind=\ind^u$ over models, by Corollary~\ref{cor: mon NIP BS-triv fork}(3)). For every small set $D$, singleton $d$ and tuples $\bar b,\bar c$, we have 
\begin{gather*}
	\left( d\ind_C D,  \ \bar{b}\ind_C D \bar c \textrm{ and } \bar c\ind_C D \bar b \right)   \ \Longrightarrow  \ d\ind_C D\bar b\ \text{ or }\ d\ind_C D \bar c.
\end{gather*}
\end{theorem}
\begin{proof}
Replacing $D$ by $CD$, we may assume $C\subseteq D$.
Suppose towards a contradiction that the three hypotheses hold, but $
 d\nind_C D \bar b$ and $d\nind_C D \bar c$. Applying BS-triviality to $\bar b\ind_C D \bar c$ and the
singleton $d$, we get that either $ \bar b d\ind_C D \bar c$ or $\bar b\ind_C D \bar c d$. The first case  contradicts $d\nind_C D \bar c$, so we must have  $\bar b\ind_C D \bar c d$. Similarly, interchanging $\bar{b}, \bar{c}$, we thus have
\begin{equation}\label{strict:relative-moved-right}
 \bar b\ind_C D \bar c d \  \textrm{ and } \  \bar c\ind_C D \bar b d.
\end{equation}

As $d\ind_C D$, by right extension we can choose a small model
$M\supseteq D$ with $d\ind_C M$. By right extension again we find $d'$ so that
\begin{equation}\label{strict:relative-dprime}
 d'\equiv_M d \ \textrm{ and } \ d'\ind_C M \bar c.
\end{equation}
In particular, $d'\ind_C D \bar c$. 
Using \eqref{strict:relative-moved-right}, we 
choose $\bar b'$ so that 
\begin{equation}\label{strict:relative-bprime-prime}
 \bar b'\equiv_{D \bar c d} \bar b \ \textrm{ and } \ \bar b'\ind_C D \bar c dd'.
\end{equation}
By \eqref{strict:relative-moved-right}, \eqref{strict:relative-bprime-prime}, the assumption $d\nind_C D\bar b$ and invariance, we have 
\begin{equation}\label{strict:relative-bprime}
 d\nind_C D \bar b' \ \textrm{ and } \ \bar c\ind_C D \bar b'd.
\end{equation}

By the first part of \eqref{strict:relative-dprime}, the second part of \eqref{strict:relative-bprime-prime} and Remark \ref{rem:invariant-transfer} (using Fact \ref{fac: props of forking etc in NIP}(3)) we have  $d\equiv_{D \bar b'}d'$. Thus, by the first part of \eqref{strict:relative-bprime} and invariance, 
\begin{equation}\label{strict:relative-dprime-forks}
 d'\nind_C D \bar b'.
\end{equation}
We apply BS-triviality over $C$ once more, now to
$\bar c\ind_C D \bar b'd$ and the singleton $d'$. The case 
$\bar cd'\ind_C D \bar b'd$ would contradict
\eqref{strict:relative-dprime-forks}. Hence we have $
 \bar c\ind_C D \bar b'dd'$, so in particular $\bar c \ind_C Ddd'$. 
By Remark \ref{rem:invariant-transfer} again (with $e:=\bar c$) we get 
$d\equiv_{D \bar c}d'$. By invariance, this gives a contradiction: $d\nind_C D \bar c$ by assumption, but $d'\ind_C D \bar c$ (by \eqref{strict:relative-dprime} and monotonicity).
\end{proof}

From now on we work with $\ind^f$ over arbitrary small sets, using Theorem~\ref{arb:forking-bs}, but the analogous results below also hold if we use $\ind^u$ over models instead (in view of the assumption of Theorem \ref{strict:relative-semilinearity} and earlier results).

\begin{defn}
	Given small sets $C, D$ and singletons $a,b$, we define $a \trianglerighteq^f_{C,D} b$ if $a \nind^f_{C} b D$. If $C = D$, we simply write $a \trianglerighteq^f_{C} b$ for $\trianglerighteq^f_{C,C}$, and just $a \trianglerighteq^f b$ when $C = \emptyset$.
\end{defn}

\begin{proposition}\label{strict:preorder}
Assume $T$ is monadically NIP and $C,D$ are arbitrary small sets. 
\begin{enumerate}
	\item The relation $\trianglerighteq^f_{C,D}$ is transitive on singletons and reflexive on $X_C := \cU \setminus \acl(C)$. Hence the relation $
 a\sim^f_{C,D}b
 : \Leftrightarrow 
\left( a\trianglerighteq^f_{C,D}b\ \text{and}\
 b\trianglerighteq^f_{C,D}a \right)$ 
is an equivalence relation on $X_C$, and $\trianglerighteq^f_{C,D}$
induces a partial order on $X_C/\sim^f_{C,D}$.

\item Let $ Y_{C,D}:=\{a\in X_C:a\ind^f_C D\}$ and 
 $Z_{C,D}:=X_C\setminus Y_{C,D}$. Then $Y_{C,D}$ is a union of $\sim^f_{C,D}$-classes, and the partial order $
 \bigl(Y_{C,D}/\sim^f_{C,D},\trianglerighteq^f_{C,D}\bigr)$  is \emph{semilinear} (i.e.~the set of predecessors of every element is linearly ordered:   for $a,b,c\in Y_{C,D}$, if $a\trianglerighteq^f_{C,D} b$ and $a\trianglerighteq^f_{C,D} c$, then $b\trianglerighteq^f_{C,D} c$ or $c\trianglerighteq^f_{C,D} b$). If $Z_{C,D}$ is nonempty, it is one equivalence class,
strictly above every class represented in $Y_{C,D}$.
\end{enumerate}

\end{proposition}
\begin{proof}
(1) Reflexivity on $X_C$: if $a \ind^f_{C} a$, in particular the formula $x = a$ does not divide over $C$, hence necessarily $a \in \acl(C)$. Transitivity: assume $a \trianglerighteq^f_{C,D} b$ and $b \trianglerighteq^f_{C,D} c$, but not $a \trianglerighteq^f_{C,D} c$. That is $a \nind^f_{C} bD$ and $b \nind^f_{C} cD$, but $a \ind^f_{C} cD$. By BS-triviality (Theorem~\ref{arb:forking-bs}), either $a b \ind^f_{C} cD$ or $a \ind^f_{C} bcD$. In either case, we get a contradiction by monotonicity. 

(2) For $a,b,c\in Y_{C,D}$, suppose $a\trianglerighteq^f_{C,D}b,c$.
If $b,c$ are incomparable, then $b\ind^f_C Dc$ and $c\ind^f_C Db$ --- contradicting 
Theorem~\ref{strict:relative-semilinearity} (with $d:=a$). This proves semilinearity  
once we verify that $Y_{C,D}$ is a union of classes.

Every $z\in Z_{C,D}$ satisfies $z\nind^f_C Da$ for every singleton
$a$, by monotonicity. Thus all elements of $Z_{C,D}$ lie in one class,
above every class of the full quotient.
For $a\in Y_{C,D}$ and $z\in Z_{C,D}$, apply the BS-triviality to
$a\ind^f_C D$ and $z$. The alternative $az\ind^f_C D$ is impossible,
so $a\ind^f_C Dz$. Hence $a$ is neither equivalent to nor above $z$.
It follows that $Y_{C,D}$ is a union of classes, and the class of
$Z_{C,D}$ is strictly greatest when it exists (in particular, the
restricted quotient is downward closed in the full quotient).
\end{proof}

\begin{remark}
	For $D=C$, the domain $Y_{C,C}$ consists of the nonalgebraic
singletons $a$ with $a\ind^f_C C$. It equals $X_C$ when $C$ is an
extension base.
\end{remark}

\subsection{Infinite chains and stability}

We show that (monadic) stability is characterized by the absence of infinite chains in the forking order. 

\begin{lemma}\label{strict:unstable-nip-chain}
Suppose $T$ is unstable and NIP. There are a small model $M$, a
singleton $d$, and an $Md$-indiscernible sequence of singletons
$(a_i:i\in\Q)$ with $a_i \notin M$, $ a_i\ind^u_M d a_{<i}$ and $a_i\nind^f_M da_j$ for all $i<j \in \Q$. In particular,  $a_i \in Y_{M,\{d\}}$, $
 a_i\trianglerighteq^f_{M,\{d\}}a_j$ and 
 $\neg\bigl(a_j\trianglerighteq^f_{M,\{d\}}a_i\bigr)$ 
 for all $i<j$.
\end{lemma}
\begin{proof}
As $T$ is unstable, there exist a model $M_0$ and a nondefinable type $q(x)\in S_1(M_0)$ (this is standard, see e.g.~\cite[Theorem~2.60]{simon2015guide}). 
Let $p(x)$ be a global extension of $q$
finitely satisfiable in $M_0$. Then $p$ is $M_0$-invariant and not
definable (indeed, if $p$ were definable, 
it would be definable over $M_0$ by $M_0$-invariance, so its restriction $q$ would be definable, a contradiction). Consider the Morley product $p^{\otimes 2}(x_0,x_1) = p(x_1) \otimes p(x_0)$ (i.e.~$p^{\otimes 2} = \tp(a,b/\cU)$ for some/any $a \models p$ and $b \models p|_{\cU,a}$, see e.g.~\cite[Section 2.2.1]{simon2015guide}; note that $\otimes$ is well-defined on invariant types, and it is not symmetric in general). Recall that a global type is \emph{generically stable} if it is both definable and finitely satisfiable (over some small model). Assuming NIP, a global invariant type $p$ is generically stable if and only if it $\otimes$-commutes with itself, i.e.~if $p(x_0) \otimes p(x_1) = p(x_1) \otimes p(x_0)$ (see e.g.~\cite[Theorem~2.29]{simon2015guide}).
Since $p$ is not definable, we can thus choose a small model $M \supseteq M_0$
and a formula $\theta(x_0,x_1)\in L(M)$ such that
\begin{equation}\label{strict:asymmetric-morley-pair}
 p^{\otimes 2}(x_0,x_1)\vdash
 \theta(x_0,x_1)\wedge\neg\theta(x_1,x_0).
\end{equation}
\noindent Take the linear order  $I := \{*\}+\Q$, and let $
 (d)+(a_i:i\in\Q)$ be a \emph{Morley sequence} of $p$ over $M$ in $\cU$ indexed by
$I$. Since $p$ is finitely satisfiable in $M_0\subseteq M$, we have $a_i\ind^u_M d a_{<i}$ for all $i$ (see e.g.~\cite[Section 2.2.1]{simon2015guide}). 
The sequence is $M$-indiscernible, so $(a_i:i\in\Q)$ is
$Md$-indiscernible. All elements of the sequence are pairwise distinct (by \eqref{strict:asymmetric-morley-pair}) and have the same type over $M$, so in particular $a_i \notin M$. Fix $i<j \in \Q$. Using \eqref{strict:asymmetric-morley-pair} we have $\models 
 \neg\theta(a_i,d) \land \theta(a_i,a_j)$, and $d\equiv_M a_j$. If $a_i\ind^f_M da_j$, then $\tp(a_i/Mda_j)$ would have a global
extension invariant over $M$ by Fact \ref{fac: forking in NIP equals inv}(1). 
 Such an extension cannot contain
both $\neg\theta(x,d)$ and $\theta(x,a_j)$, as $d\equiv_M a_j$.
Therefore $a_i\nind^f_M da_j$.
Conversely, $a_j\ind^u_M da_i$ by the above, so in particular $a_j\ind^f_M da_i$.
\end{proof}

\begin{theorem}\label{strict:instability-characterization}
For a monadically NIP theory $T$, the following are equivalent:
\begin{enumerate}
\item $T$ is unstable.
\item For some small sets $C,D$, the partial order 
$Y_{C,D}/\sim^f_{C,D}$ is not an antichain.
\item For some small sets $C,D$, the partial order $Y_{C,D}/\sim^f_{C,D}$  has an infinite chain of distinct elements.
\item For some small model $M$ and singleton $d$, the partial order 
$Y_{M,\{d\}}/\sim^f_{M,\{d\}}$ contains a strict chain
$([a_i]:i\in\Q)$, with $[a_i]\trianglerighteq^f_{M,\{d\}}[a_j]$
for all $i<j$.
\end{enumerate}
\end{theorem}
\begin{proof}
Lemma~\ref{strict:unstable-nip-chain} gives $(1)\Rightarrow(4)$,
and $(4)\Rightarrow(3)\Rightarrow(2)$ is immediate.

For $(2)\Rightarrow(1)$, suppose $T$ is stable. Then, using symmetry of $\ind^f$ and left total triviality (Corollary~\ref{cor: arbitrary-left-total}), hence we also have right total triviality, for any $a,b\in Y_{C,D}$ we have 
\[
 a\ind^f_C Db\quad\Longleftrightarrow\quad a\ind^f_C b
 \quad\Longleftrightarrow\quad b\ind^f_C a
 \quad\Longleftrightarrow\quad b\ind^f_C Da.
\]
Thus the restricted preorder is symmetric, so its quotient is an
antichain. This contradicts (2). \end{proof}

\begin{remark}
	Note that $a \ind^f_{\emptyset} b$ for any singletons $a \neq b$ in DLO, hence $\trianglerighteq^f_{\emptyset}$ is an antichain. This shows that the introduction of an additional parameter $d$ in $\trianglerighteq^f_{\emptyset, d}$ is necessary to detect the failure of stability.
\end{remark}

\section{$\omega$-categorical monadic NIP theory with no finite extension bases}
\subsection{Trees with circularly ordered open cones}\label{sec: Trees with circularly ordered open cones}

\begin{defn}
	By a \emph{meet tree} we mean a structure $(T, \leq, \land)$ so that $\leq$ is a partial order such that for every $x$, the set $\{y \in T : y \leq x \}$ is linearly ordered and for any $x,y \in T$, $\{z \in T : z \leq x,y \}$ has a greatest element $x \land y$.
\end{defn}

\begin{fact}(See \cite[Fact 4.5]{estevan2021non}, \cite[Section 2.3.1]{simon2015guide} or \cite[Lemma 3.14]{chernikov2022semi}.)
	The (universal) theory of meet trees $T_{\TTr}$ has a model completion $T_{\DT}$ which is determined by saying that: $(\leq, \land)$ is a meet tree; for every $x$, the set $\{y : y \leq x\}$ is a dense linear order with a maximal element and no minimal one; for every $x$ there are infinitely many open cones above $x$. The theory $T_{\DT}$ is complete, $\omega$-categorical and has quantifier elimination; it is the theory of the Fra\"iss\'e limit of the class of finite meet trees.
\end{fact}

\begin{defn}
	For $a,b>c$, we write $a\sim_{c}b$ if $a\land b>c$, note that for
each $c$ this is an equivalence relation, with classes corresponding
to the open cones above $a$.
\end{defn}

We recall that a Fra\"iss\'e class $\mathcal{K}$ satisfies the \emph{Strong Amalgamation Property}, or \emph{SAP}, if: for every $A,B,C \in \mathcal{K}$ and embeddings $f_1: A \to B, f_2: A \to C$ there exist $D \in \mathcal{K}$ and embeddings $g_1: B \to D, g_2: C \to D$ so that $g_1 \circ f_1 = g_2 \circ f_2$, and $\im(g_1) \cap \im(g_2) = \im(g_2 \circ f_2)$ (which is $= \im(g_1 \circ f_1)$); and in JEP, the images of $A,B$ in $C$ are disjoint.
\begin{fact}
	The (universal) theory $T_{\TCirc}$ of circular orders  in the language $L = \{C(x,y,z)\}$ has a model completion $T_{\DCirc}$ of dense circular orders. It is the theory of the Fra\"iss\'e limit of the class of finite circular orders (= finite models of $T_{\TCirc}$), which satisfies SAP. It is $\omega$-categorical and has quantifier elimination.
\end{fact}

Now we consider the theory of dense meet trees so that for every element $x$, the set of open cones above $x$ is equipped with a dense circular order. This is a special case of the more general construction in \cite[Section 4.3]{estevan2021non}:
\begin{defn}
	Let $L^{\ast} := \{\leq, \land, C^{\ast}(x, y_1, y_2, y_3) \}$, and consider the theory $T^{\ast}_{\forall}$ consisting of the axioms $T_{\TTr}$ for meet trees and:
	\begin{enumerate}
		\item $\forall x, y_0, y_1, y_2, y'_0, y'_1, y'_2 \left( C^{\ast}(x, y_0, y_1, y_2) \land \bigwedge_{i<3}  y_i \sim_{x} y'_i\right) \rightarrow C^{\ast}(x, y'_0, y'_1, y'_2)$,
		\item $\forall x, y_0, y_1, y_2  \ C^{\ast}(x,y_0,y_1,y_2) \rightarrow \bigwedge_{i<3} x < y_i$,
		\item for every $x$, $C^{\ast}(x, -, -, -)$ is a circular order on the set of open cones above $x$.
	\end{enumerate}
\end{defn}

\begin{fact}
	The class of finite models of $T^{\ast}_{\forall}$ is a uniformly locally finite Fra\"iss\'e class (\cite[Proposition 4.9]{estevan2021non}), and let $T^{\ast}$ be the theory of the limit. Then $T^{\ast}$ is $\omega$-categorical, NIP, has quantifier elimination, and is the model completion of $T^{\ast}_{\forall}$ (\cite[Corollary 4.11]{estevan2021non}). For every $M \models T^{\ast}$ and $c \in M$,  we consider the $L$-structure $M_c := \{a / \sim_{c} : a \in M, a > c\}$ with $C^{M_c}(x/\sim_c, y/ \sim_c, z/\sim_c) :\Leftrightarrow \left( x,y,z > c \land (C^{\ast})^{M}(c, x,y,z) \right)$. Then $M_c \models T_{\DCirc}$ (\cite[Proposition 4.13]{estevan2021non}).
\end{fact}

The proof below will show that every \emph{non-empty} finite set in $T^{\ast}$ is not an $\ind^{f}$-extension base. To ensure that $\emptyset$ is also not an $\ind^{f}$-extension base, we will additionally name a single constant. Namely, we choose a model $M\models T^*$ and any element $r\in M$. We let $\wideT \eqdef \Th(M,r)$ in the expanded language $\widehat L=\{\leq,\wedge,C^*,r\}$.  Note that still $\wideT$ eliminates quantifiers in $\widehat L$, is NIP and $\omega$-categorical. We work in a monster model $\cU \models \wideT$.

\begin{lemma}\label{lem:lifting}
Let $A$ be finite, $B=\langle A\cup\{r\}\rangle$, and let $a\in B$ be maximal (in the tree order). Suppose $\bar u=(u_1,\ldots,u_n), \bar v=(v_1,\ldots,v_n)$ 
are tuples in $\cU$ such that $u_i,v_i > a$, the $u_i$ lie in pairwise distinct $\sim_a$-classes, and the $v_i$ lie in pairwise distinct $\sim_a$-classes. If
\[
    \qftp_{\circleT}\big(u_1/\sim_a,\ldots,u_n/\sim_a\big)
    =
    \qftp_{\circleT}\big(v_1/\sim_a,\ldots,v_n/\sim_a\big)
\]
over the empty set in the circular order $\cU_a$, then $\bar u\equiv^{\widehat L}_B \bar v$.  In particular, every empty-set indiscernible sequence of distinct cone-classes in $\cU_a$ lifts to a $B$-indiscernible sequence of representatives in $\cU$.
\end{lemma}

\begin{proof}
By quantifier elimination for $T^*$, it is enough to compare quantifier-free types over $B$. If $u_i>a$ and $b\in B$, then $ u_i\wedge b = a\wedge b$ by maximality of $a$.
 Also, if $i\neq j$, then $u_i\wedge u_j=a$, because $u_i$ and $u_j$ are in distinct open cones above $a$. The same statements hold for the $v_i$. Thus the quantifier-free tree type over $B$ is the same for $\bar u$ and $\bar v$.

Let now $c$ be a possible center occurring in the finite substructure generated by $B\bar u$, which is just $B\cup\{u_1,\ldots,u_n\}$ by the above. If $c=a$, $C^*(a;y_1,y_2,y_3)$ can only hold for $y_i \in B\cup\{u_1,\ldots,u_n\}$ above $a$, so  among the $u_i$. In that case its truth is exactly the truth of the corresponding circular-order relation in $\cU_a$. By assumption the tuples of cone-classes of $\bar u$ and $\bar v$ have the same quantifier-free circular-order type. If $c<a$, then all elements $u_i$ lie in the same open cone above $c$, namely the cone containing $a$. Thus replacing $u_i$ by $v_i$ does not change any $\sim_c$-class appearing over $B$, and it does not change the truth of any $C^*$-formula with center $c$. If $c\in B$ and $c\not\leq a$, then $c<u_i$ is impossible. Hence any $C^*$-formula with center $c$ and one of the $u_i$ as an argument is false, and the same holds for the $v_i$. Finally, if $c=u_i$ for some $i$, then none of the generated elements lies strictly above $u_i$, hence every $C^*$-formula with center $u_i$ and arguments from the generated substructure is false, and the same holds for the corresponding $v_i$.
\end{proof}

\begin{prop}
No finite subset of $\cU \models \wideT$ is an $\ind^f$-extension base.
\end{prop}
\begin{proof}
Let $A\subseteq\cU$ be finite (possibly empty). Let $B := \langle A\cup\{r\}\rangle \neq \emptyset$, and choose $a\in B$ maximal in the tree order. Then $a\in\dcl_{\widehat L}(A)$, thus the formula $a < x$ does not divide over $A$. We show that $a<x$ forks over $A$. Choose any $u,v>a$ in distinct open cones above $a$. Then
\begin{gather*}
	x > a \  \vdash \  (x \sim_{a} u) \,  \lor \, (x \sim_{a} v) \, \lor \,  C^{\ast}(a, u,x,v) \,  \lor \,  C^{\ast}(a, v, x, u),
\end{gather*}
and we claim that each of the disjuncts divides over $A$ (and $B$).

Working in $M_a \models T_{\DCirc}$, we can choose an $\emptyset$-indiscernible sequence of pairwise distinct elements $(u_i/\sim_{a})_{i < \omega}$ with $u_0/ \sim_{a} = u/ \sim_{a}$. Then, by Lemma \ref{lem:lifting}, $(u_i)_{i <\omega}$ with $u_0 = u$ is a $B$-indiscernible sequence in $\cU$, and $u_i \perp u_j$ for all $i \neq j$. As $\{x > u_i : i < \omega\}$ is $2$-inconsistent, $x \sim_{a} u$ divides over $B$.

Similarly, in any model of $T_{\DCirc}$, given any two points $p \neq q$, there is an $\emptyset$-indiscernible sequence of pairs $(p_i, q_i)_{i \in \omega}$ with $(p_0,q_0) = (p,q)$ and so that $\{C(p_i, x, q_i)\}_{i <\omega}$ is $2$-inconsistent (and all elements among $p_i,q_i$ are pairwise distinct). By Lemma \ref{lem:lifting} again, we can thus find a $B$-indiscernible sequence of pairs $(u_i, v_i)_{i < \omega}$ with $u_i,v_i > a$ in $\cU$ so that $(u_0,v_0) = (u,v)$ and $\{C(u_i / \sim_{a}, x, v_i/\sim_a) : i < \omega\}$ is $2$-inconsistent in $\cU_{a}$, hence $\{C^{\ast}(a, u_i, x, v_i) : i < \omega\}$ is $2$-inconsistent in $\cU$. It follows that $C^{\ast}(a, u,x,v) $ divides over $B$.
\end{proof}

\subsection{Convexly ordered colored meet trees}\label{sec: convexly ordered trees}

We will show that the theory $\wideT$ constructed in Section \ref{sec: Trees with circularly ordered open cones} is in fact monadic NIP, hence dp-minimal, and distal. We take an opportunity to prove a more general result about linearly ordered expansions of \emph{arbitrary} meet trees first.

We let $(T, <, \land, (P_i)_{i \in I})$ be a \emph{colored meet tree}, i.e.~a meet tree  expanded by arbitrary unary predicates $P_i$. It is proven in \cite{parigot1982theories} that the theory of any meet tree is monadically NIP, hence any colored meet tree is dp-minimal. Dp-minimality of colored trees is reproved explicitly via an analysis of indiscernible sequences in \cite[Section 4]{simon2011dp}.

\begin{defn}
\begin{enumerate}
\item We add a ternary relation $x\preceq_{z}y$ such that for each $z$,
it defines a linear ordering on the set of all open cones above $a$.
\item For each $a\in T$, let $S_{a}$ be a new sort of all open cones above
$a$, equipped with unary predicates induced by all unary $a$-definable
sets and the linear ordering $\preceq_{a}$. We denote by $\tp_{S_{a}}$
the type of a tuple of cones in $S_{a}$ in this language.
\end{enumerate}
\end{defn}

\begin{remark}\label{rem: def lin ord in convex tree}
	Note that a global linear order $\prec$ on $T$ extending the partial tree order $<$ 
is definable from $<$ and $x\preceq_{z}y$: we let $a \prec b$ if either $a < b$, or $a\perp b$ and $a\prec_{a\land b}b$. In particular, every open cone is convex with respect to $\prec$.
\end{remark}

We will show that every structure $(T, <, \land, (P_i)_{i \in I}, x \prec_{z} y)$ is dp-minimal, hence also monadically NIP since any expansion of such structure by unary predicates is also of the same form (note that we are not in the context of \cite{estevan2021non} here since nothing is assumed to be $\omega$-categorical).   
We adapt the analysis in \cite[Section 4]{simon2011dp}, combining the cases
of trees and linear orders. When we talk about types, we mean types in the full language $\{ <, \land, (P_i)_{i \in I}, x \prec_{z} y \}$ unless explicitly specified.

\begin{lemma}
	\label{lem:=000020comp}Assume that $a\in T$, $\bar{b}$ is a tuple
from the closed cone above $a$, and $\bar{c}$ is a tuple from its
complement. Then $\tp\left(\bar{b}/a\right)\cup\tp\left(\bar{c}/a\right)\vdash\tp\left(\bar{b},\bar{c}/a\right)$.
\end{lemma}

\begin{proof}
By back-and-forth. Note that $\tp\left(\bar{b}/a\right)=\tp\left(\bar{b}'/a\right) 
\, \land \, \tp\left(\bar{c}/a\right)\cup\tp\left(\bar{c}'/a\right)$
implies that $\qftp\left(\bar{b},\bar{c}/a\right)=\qftp\left(\bar{b}',\bar{c}'/a\right)$.
If $d$ is a new element in the closed cone above $a$, let $d'$
be such that $\bar{b} d\equiv_{a}\bar{b}'d'$. If $d$ is in its complement,
take $d'$ such that $\bar{c} d \equiv_{a}\bar{c}'d'$.
\end{proof}
\begin{lemma}
\label{lem:=000020in=000020cones}The $\tp_{S_{a}}$-type of a tuple
$\bar{b}=\left(b_{1},\ldots,b_{n}\right)$ in $S_{a}$ is determined
by 
$$\bigcup_{b_{i}\preceq_{a}b_{j}}\tp_{S_{a}}\left(b_{i},b_{j}\right).$$
\end{lemma}

\begin{proof}
By the corresponding fact for colored linear orders (see \cite[Proposition 4.2]{simon2011dp}).
\end{proof}

\begin{prop}
\label{prop:=000020back-and-forth}Let $A=\left(a_{i}\right),B=\left(b_{i}\right)$
be two finite substructures of $T$ (i.e., they are $\land$-closed).
Then $\tp\left(A\right)=\tp\left(B\right)$ if and only if:
\begin{enumerate}
\item $A$ is isomorphic to $B$ as a substructure,
\item $\tp\left(a_{i},a_{j}\right)=\tp\left(b_{i},b_{j}\right)$ for any
$i,j$ such that $a_{i}\geq a_{j}$,
\item for any $a_{i},a_{j}>a_{k}$ such that $a_{i}\land a_{j}=a_{k}$ and
$a_{i}\succeq_{a_{k}}a_{j}$, we have $\tp_{S_{a_{k}}}\left(c_{i},c_{j}\right)=\tp_{S_{b_{k}}}\left(d_{i},d_{j}\right)$
with $c_{i},d_{i}$ the corresponding cones in $S_{a_{k}},S_{b_{k}}$
respectively.
\end{enumerate}
\end{prop}

\begin{proof}
By back-and-forth. Given $A,B$ satisfying (1)\textendash (3) (which
implies that they have the same quantifier-free type), we want to
add a new element $a$. Without loss of generality $A\cup\left\{ a\right\} $ is a substructure
(otherwise first add the meets of $a$ with the elements of $A$).
Then one of the following cases must occur.

\noindent (i). The element $a$ is $\leq$-below all elements in $A$.\\
Let $a_{0}\in A$ be $\leq$-minimal. Let $b$ be such that $a_{0}a\equiv b_{0}b$.
For any $i$ we have $\tp\left(a_{i},a_{0}\right)=\tp\left(b_{i},b_{0}\right)$
and $\tp\left(a,a_{0}\right)=\tp\left(b,b_{0}\right)$. Then $\tp\left(a_{i},a\right)=\tp\left(b_{i},b\right)$
by Lemma \ref{lem:=000020comp}. Conditions (1) and (3) are obvious.

\noindent (ii). The element $a$ is $\leq$-above some element in $A$, say $a_{1}$,
so that the open cone $c$ above $a_{1}$ containing $a$ contains no
elements from $A$.

(a). Assume there are some $a_{2},a_{3}>a_{1}$ in $A$ such that $a_{2}\land a_{3}=a_{1}$
and $a_{2}\preceq_{a_{1}}a\preceq_{a_{1}}a_{3}$. Let $a_{2},a_{3}$
be the $\preceq_{a_{1}}$-maximal such and the $\preceq_{a_{1}}$-minimal
such, respectively, and let $c_{2},c_{3},d_{2},d_{3}$ be the corresponding
cones for the $a$'s and $b$'s. By assumption, let $d\in S_{b_{1}}$
be such that $\tp_{S_{a_{1}}}\left(c_{2},c_{3},c\right)=\tp_{S_{b_{1}}}\left(d_{2},d_{3},d\right)$.
Let $b\in d$ be such that $\tp\left(a/a_{1}\right)=\tp\left(b/b_{1}\right)$
(exists since $S_{a_{1}}$ is equipped with the full induced unary
structure definable over $a_{1}$, as $\exists a\left(\pi_{S_{a_{1}}}\left(a\right)=c\land\psi\left(a\right)\right)\in\tp_{S_{a_{1}}}\left(c\right)$
for every $\psi\in\tp_{S_{a_{1}}}\left(a\right)$, and $d$ has the
same type in $S_{b_{1}}$). Now if $a_{i}\leq a_{1}$, we have $a_{i}a_{1}\equiv b_{i}b_{1}$
and $aa_{1}\equiv bb_{1}$, so by Lemma \ref{lem:=000020comp} we have $a_{i}a\equiv b_{i}b$.
If $a_{i}\land a=a_{1}$, then either $a_{i}\preceq_{a_{1}}a_{2}$
or $a_{i}\succeq_{a_{1}}a_{3}$. Say, in the first case, we have by
assumption $\tp_{S_{a_{1}}}\left(c_{i},c_{2}\right)=\tp_{S_{b_{1}}}\left(d_{i},d_{2}\right)$
and $\tp_{S_{a_{1}}}\left(c,c_{2}\right)=\tp_{S_{b_{1}}}\left(d,d_{2}\right)$,
hence $\tp_{S_{a_{1}}}\left(c_{i},c\right)=\tp_{S_{b_{1}}}\left(d_{i},d\right)$
by Lemma \ref{lem:=000020in=000020cones}. Hence (1)\textendash (3)
hold.

(b). There is only one $a_{2}>a_1\in A,a_{2}\neq a$ such that $a_{2}\land a=a_{1}$.
Say $a\preceq_{a_{1}}a_{2}$. Then we take $b$ such that $a_{1}a_{2}a\equiv b_{1}b_{2}b$,
and argue similarly (note that this implies $\tp_{S_{a_{1}}}\left(c_{2}c\right)=\tp_{S_{b_{1}}}\left(d_{2}d\right)$
in particular ).

(c). There is no $a_{2}>a_1\in A,a_{2}\neq a$ such that $a_{2}\land a=a_{1}$.
Then we take $b$ such that $\tp\left(a_{1},a\right)=\tp\left(b_{1},b\right)$
and argue similarly.

\noindent (iii). The point $a$ is $\leq$-between two points in $A$, say $a_{0}\leq a_{1}$, so that there are no points of $A$ which are $\leq$-between $a_{0}$ and $a_{1}$.\\
Take $b$ be such that $a_{0}a_{1}a\equiv b_{0}b_{1}b$. Then (1)
holds. If $i$ is such that $a_{i}>a$, then $a_{i}\ge a_{1}$, and
$a_{i}a\equiv b_{i}b$ by Lemma \ref{lem:=000020comp}. Similarly
for $a_{i}<a$, so (2) holds. If $a_{i}\in A$ is such that $a_{2}:=a\land a_{i}\leq a_{0}$,
we have that $a$ and $a_{1}$ are in the same cone over $a_{2}$,
so $\tp_{S_{a_{2}}}\left(cc_{i}\right)=\tp_{S_{a_{2}}}\left(c_{1}c_{i}\right)=\tp_{S_{b_{2}}}\left(d_{1}d_{i}\right)=\tp_{S_{b_{2}}}\left(dd_{i}\right)$,
so (3) holds as well.
\end{proof}

\begin{cor}
The full induced structure on $S_{a}$ is just that of a linear order
with colors induced by all unary $a$-definable sets.
\end{cor}

\begin{cor}
\label{cor:=000020det=000020by=0000203-types}If $A\subseteq T$,
then $\bigcup_{\left(a,b,c\right)\in A^{3}}\tp\left(a,b,c\right)\vdash\tp\left(A\right)$.
\end{cor}

\begin{proof}
Let $A_{0}$ be the substructure of $T$ generated by $A$. By Proposition
\ref{prop:=000020back-and-forth}, the type of $A_{0}$ is implied
by:
\begin{itemize}
\item $\qftp\left(A_{0}\right)$,
\item the set of $2$-types $\tp\left(a,b\right)$ for $\left(a,b\right)\in A_{0}^{2}$
with $a<b$,
\item the set of $2$-types $\tp_{S_{c}}\left(c_{a},c_{b}\right)$ for $a,b,c\in A_{0}^{3}$
with $a\perp b$ and $a\land b=c$.
\end{itemize}
Without loss of generality $A$ is finite. Then the isomorphism type of $A_{0}$ as a tree
is determined by the set of $3$-types of elements in $A$. And for
every $m_{1}\perp m_{2},m_{1}\land m_{2}=m_{3}$ in $A_{0}$, there
are some $a_{1},a_{2}\in A$ such that $a_{i}$ is in the same open
cone over $m_{3}$ as $m_{i}$, so the local linear orders type of
$A_{0}$ is also determined. Given any $m_{1}\leq m_{2}$ in $A_{0}$,
they are both definable from some $3$ points in $A$, so their type
is determined by a $3$-type from $A$. Given any $m_{1}\perp m_{2},m_{1}\land m_{2}=m_{3}$
in $A_{0}$, there are some $a_{1},a_{2}\in A$ such that $a_{i}$
is in the same open cone over $m_{3}$ as $m_{i}$, and $m_{3}$ is
definable from $a_{1},a_{2}$, so $\tp_{S_{m_{3}}}\left(c_{m_{1}},c_{m_{2}}\right)$
is determined by $\tp\left(a_{1}a_{2}\right)$.
\end{proof}
\begin{theorem}\label{thm: dp-min of conv ord trees}
The theory of $(T, <, \land, (P_i)_{i \in I}, x\preceq_{z}y)$ is dp-minimal, monadic NIP and distal.
\end{theorem}

\begin{proof}
First we prove dp-minimality. Let $\left(a_{i}\right)_{i\in I},\left(b_{j}\right)_{j\in J}$ be
mutually indiscernible sequences with $I,J$ dense without endpoints,
and $\alpha\in T$ an element. 

Assume first that $a_{i},b_{j}$ are singletons. It is easy to see
that every non-constant indiscernible sequence must satisfy one of
the following:
\begin{enumerate}
\item The sequence $\left(a_{i}\right)$ is $\leq$-monotonous (increasing
or decreasing).
\item The elements $a_{i}$ are pairwise $\leq$-incomparable, and $a_{i}\land a_{j}$
is constant equal to some $\beta$, and $a_{i}$ is $\preceq_{\beta}$-monotonous.
\item The elements $a_{i}$ are $\leq$-incomparable and $a_{i}\land a_{j},i<j$
depends only on $i$. Then let $a_{i}':=a_{i}\land a_{j}$ (for some/any
$j>i$). Then $\left(a_{i}'\right)$ is an $\leq$-increasing indiscernible
sequence.
\item The elements $a_{i}$ are $\leq$-incomparable and $a_{i}\land a_{j},i<j$
depends only on $j$. Then let $a_{j}':=a_{i}\land a_{j}$ (for some/any
$i<j$) is an $\leq$-decreasing indiscernible sequence.
\end{enumerate}
Assume $\left(a_{i}\right)$ is in case 1. We say that $\alpha$ does not
cut $\left(a_{i}\right)$ if $\left\{ x:x\leq\alpha\right\} $ either
contains the sequence, or is disjoint from it. In case 2, we say that
$\alpha$ doesn't cut $\left(a_{i}\right)$ if either $\beta\not\leq\alpha$,
or $\beta<\alpha$ and $\left\{ x:x\preceq_{\beta}\alpha\right\} $
either contains the sequence, or is disjoint from it. In either case,
if $\alpha$ does not cut $\left(a_{i}\right)$, then it is indiscernible
over $\alpha$, by Lemma \ref{lem:=000020comp} in case 1 and by Lemma
\ref{prop:=000020back-and-forth} in case 2. E.g., in case 2 with
$a_{i}$ $\preceq_{\beta}$-increasing and $\alpha$ such that $a_{i}\preceq_{\beta}\alpha$
for all $i$, let $i_{1}<\ldots<i_{n},j_{1}<\ldots<j_{n}$ and $k\in I$
greater than all these indices. Clearly $\left(a_{i}\right)$ is indiscernible
over $\beta$ since $\beta$ is definable from any two elements, so
$a_{i}\equiv_{\beta}a_{j}$, and for the cones we have $\tp_{S_{\beta}}\left(c_{i_{t}},c_{\alpha}\right)=\tp_{S_{\beta}}\left(c_{j_{t}},c_{\alpha}\right)$
for all $1\leq t\leq n$ by Lemma \ref{lem:=000020in=000020cones},
as $\tp_{S_{\beta}}\left(c_{i_{t}}/c_{k}\right)\cup\tp_{S_{\beta}}\left(c_{\alpha}/c_{k}\right)\vdash\tp\left(c_{i_{t}},c_{\alpha}/c_{k}\right)$.
Thus Lemma \ref{prop:=000020back-and-forth} applies.

In cases 3 and 4, if $\left(a_{i}\right)$ is $\alpha$-indiscernible,
then it is also the case for $\left(a_{i}'\right)$. Conversely, if
$\left(a_{i}'\right)$ is $\alpha$-indiscernible, then $\alpha$
does not cut $\left(a'_{i}\right)$ and it is easy to check using Lemmas
\ref{lem:=000020comp} and \ref{prop:=000020back-and-forth} that
$\left(a_{i}\right)$ is indiscernible over $\alpha$. So we can replace
$\left(a_{i}\right)$ by $\left(a_{i}'\right)$ which is in case 1.

So to verify dp-minimality in this case, assume that $\left(a_{i}\right),\left(b_{j}\right)$
are mutually indiscernible sequences each in case 1 or 2. Then $\alpha$
cannot cut both of them \textemdash{} which implies that at least
one of the sequences is indiscernible over $\alpha$ by the previous
paragraph. For example, assume $\left(a_{i}\right)$ is in case 1,
$\leq$-increasing and $\left(b_{j}\right)$ is in case 2, let $\beta:=b_{i}\land b_{j}$.
If $\alpha$ cuts $\left(b_{j}\right)$, then $\alpha>\beta.$ As
$\left(a_{i}\right)$ is $\beta$-indiscernible, $\beta$ does not
cut $\left(a_{i}\right)$. Then the only way for $\alpha$ to cut
$\left(a_{i}\right)$ is if $\beta<\alpha_{i}$ for all $i$ and $\alpha$
and $\alpha_{i}$ are in the same open cone over $\beta$, but that
would contradict mutual indiscernibility.

Now we reduce the case of general sequences of tuples to the previous
case of singletons. 

Let $\left(a_{i}\right)$ be an indiscernible sequence of $n$-tuples
such that $\left(a_{i}\right)$ is not $\alpha$-indiscernible, then
there is some indiscernible sequence $\left(d_{i}\right)$ of singletons
with $d_{i}\in\dcl\left(a_{i}\right)$ such that$\left(d_{i}\right)$
is not $\alpha$-indiscernible.

By Corollary \ref{cor:=000020det=000020by=0000203-types}, without loss of generality $n=2$,
write $a_{i}=\left(b_{i},c_{i}\right)$ and $m_{i}:=b_{i}\land c_{i}$.
We claim that if each of $\left(b_{i}\right),\left(c_{i}\right),\left(m_{i}\right)$
is indiscernible over $\alpha$, then $\left(a_{i}\right)$ is also
indiscernible over $\alpha$ \textemdash{} be the same analysis of
cases as in \cite{simon2011dp}, using Proposition \ref{prop:=000020back-and-forth}.

Monadic NIP follows since any expansion of $(T, <, \land, (P_i)_{i \in I}, x\preceq_{z}y)$ by unary predicates is again a structure of the same form (with different $I$), hence also dp-minimal.

Finally, distality follows using \cite[Lemma 2.10]{simon2013distal} since the theory is dp-minimal and any infinite indiscernible sequence
is not totally indiscernible because of the definable total ordering (Remark \ref{rem: def lin ord in convex tree}).
\end{proof}

\begin{cor}
Any colored tree  (or colored meet tree) has a dp-minimal distal monadically NIP expansion.
\end{cor}
\begin{proof}
	Immediate by Theorem \ref{thm: dp-min of conv ord trees}, as any tree embeds into a meet tree, hence any colored tree is definable on singletons in a colored meet tree.
\end{proof}

\begin{cor}
The theory $\wideT$ from Section \ref{sec: Trees with circularly ordered open cones} is monadically NIP and distal.
\end{cor}
\begin{proof}
	
	Given a model of $M \models \wideT$, it is definable on singletons in a structure $(T, <, \land, x\preceq_{z}y)$, where $(T, <, \land)$ is the underlying tree and  we  choose $x\preceq_{z}y$ to be a dense linear order on the open cones over $z$ so that $C^{\ast}(e, a,b,c) := (a \prec_{e} b \prec_{e}  c) \lor (b \prec_{e}  c \prec_{e}  a) \lor (c \prec_{e}  a \prec_{e}  b)$ for all $e,a,b,c$. It follows by Theorem \ref{thm: dp-min of conv ord trees} that $\wideT$ is monadically NIP, in particular dp-minimal. Also note that  any non-constant indiscernible sequence in a model of $\wideT$ is not totally indiscernible (it is enough to check it for sequences of singletons, and similarly to the cases (1)--(4) in the proof of Theorem \ref{thm: dp-min of conv ord trees}, failure of total indiscernibility is given by $<$ or by $C^{\ast}$). Hence $\wideT$ is distal using  \cite[Lemma 2.10]{simon2013distal}.
\end{proof}

\bibliographystyle{plain}
\bibliography{ref}

\end{document}